\documentclass[12pt,amssymb]{amsart}
  \usepackage{amssymb}
  \usepackage[dvips]{graphicx}
\usepackage[all]{xy}
\usepackage[mathscr]{eucal}

\newcommand{\ra}{\rangle}
\newcommand{\la}{\langle}

\newcommand{\F}{{\mathbb F}}
\newcommand{\Z}{{\mathbb Z}}
\newcommand{\PP}{\mathbb{P}}

\newcommand{\C}{{\mathbb C}}

\newcommand{\da}{\dot{\alpha}}

\newcommand{\db}{\dot{\beta}}

\newcommand{\md}{{\rm mod}\ }

\newcommand{\Int}{\mathrm{Int}\: }
\newcommand{\aut}{\mathrm{Aut}}

\newcommand{\Ga}{\mathrm{Gal}}
\newtheorem{thm}{Theorem}[section]
\newtheorem{lemma}[thm]{Lemma}
\newtheorem{prop}[thm]{Proposition}
\newtheorem{cor}[thm]{Corollary}
\newtheorem{Remark}[thm]{Remark}

\newcommand{\tz}{\tilde{z}}

\newcommand{\Aut}{\mathrm{Aut}\:}
\newcommand{\Out}{\mathrm{Out}\:}
\newcommand{\In}{\mathrm{Int}\:}

  \newcommand{\cput}[3]{%
    \makebox[0pt]{%
      \null\kern #1\raisebox{#2}[0pt][0pt]{\ensuremath{#3}}\hss
    }
  }
  \newcommand{\insertgraph}[1]{%
    \raisebox{-0.5\height}{
      \includegraphics[scale=1.0]{graph.#1}
    }%
  }

  \newcommand{\xgen}{\tilde{a}}
  \newcommand{\ygen}{\tilde{b}}
  \newcommand{\agen}{\tilde{w}}
  \newcommand{\bgen}{\tilde{u}}
  \newcommand{\cgen}{\tilde{v}}
  \newcommand{\ainv}{\agen^{-1}}
  \newcommand{\binv}{\bgen^{-1}}
  \newcommand{\cinv}{\cgen^{-1}}

  \newcommand{\BlueOne}{0}
  \newcommand{\BlueTwo}{1}
  \newcommand{\Green}{2}
  
  \newcommand{\GroupGenerated}[1]{\langle #1\rangle}

  \newcommand{\Sphere}[1]{\Sigma_0^#1}
  \newcommand{\Conf}[1]{\operatorname{Conf}_2(\Sphere{#1})}
  \newcommand{\SurfBot}{\pi_1(0 , 3)}
  \newcommand{\SurfTop}{\pi_1(0 , 4)}
  \newcommand{\LoopBot}{\gamma}
  \newcommand{\LoopTop}{\delta}

  \newcommand{\BasePt}{p}

  \newcommand{\acts}[2]{
    #1(#2)
  }

  \newcommand{\EllCurve}[1][*]{E_{#1}}
  \newcommand{\EllFiber}[1][*]{\EllCurve[#1]\setminus \{O\}}

  \newcommand{\firstdown}{\hat{a}}
  \newcommand{\seconddown}{\hat{b}}
  \newcommand{\firstup}{\hat{x}}
  \newcommand{\secondup}{\hat{y}}

  \newcommand{\DownSpace}{\PP^1 \setminus \{0, 1, \infty\}}
  \newcommand{\DownGroup}{\pi_1(\DownSpace)}
  
  \newcommand{\UpGroup}{\pi_1(\EllFiber[2])}

  \newcommand{\DownPath}{\gamma}
  \newcommand{\UpPath}{c}

  \newcommand{\FirstHomologyOf}[1]{\operatorname{H}_1(#1;\Z)}
  \newcommand{\FirstHomology}{\FirstHomologyOf{\EllFiber[2]}}

  \newcommand{\TheCircle}{S^1}
  \newcommand{\TheInterval}{[0,1]}
  \newcommand{\TheAnnulus}{\TheCircle\times\TheInterval}
  \newcommand{\TheHomotopy}{H}
  \newcommand{\TotalSpace}{S}
  \newcommand{\BaseSpace}{\PP^1 \setminus \{0, 1, \infty\}}
  \newcommand{\FiberSpace}{\EllFiber[2]}

  \newcommand{\concat}{\circ}
\begin{document}
\title[A group-theoretic proof of Asada's theorem]{The congruence subgroup
property for $\Aut F_2:$ \\ A group-theoretic proof of Asada's theorem}
\author[Bux]{Kai-Uwe Bux}
\address{Kai-Uwe Bux, Fakult\"at f\"ur Mathematik, Universit\"at Bielefeld, Postfach 100131,
33501 Bielefeld, Germany}
\email{bux\_\!\_\,2009@kubux.net}
\author[Ershov]{Mikhail V. Ershov}
\thanks{The second-named author would like to acknowledge partial support
from the {\small NSF} grant {\small DMS-0901703}.} 
\address{Mikhail Ershov, Department of Mathematics, University of Virginia,
Charlottesville, VA 22904, USA}
\email{ershov@virginia.edu}
\author[Rapinchuk]{Andrei S. Rapinchuk}
\thanks{The third-named author would like 
to acknowledge partial support from the {\small NSF} grant {\small DMS-0965758} and the Humboldt Foundation.}
\address{Andrei Rapinchuk, Department of Mathematics, University of Virginia,
Charlottesville, VA 22904, USA}
\email{asr3x@virginia.edu}


\thanks{First published in Groups Geom. Dyn. 5 (2011), no. 2, published by the European Mathematical Society}
\thanks{\copyright European Mathematical Society}

\maketitle

\centerline{\it To Fritz Grunewald}
\begin{abstract}
The goal of this paper is to give a group-theoretic proof
of the congruence subgroup property for $\Aut(F_2)$, the group of
automorphisms of a free group on two generators. This result was first
proved by Asada using techniques from anabelian geometry, and our
proof is, to a large extent, a translation of Asada's proof into
group-theoretic language. This translation enables us to simplify
many parts of Asada's original argument and prove a quantitative
version of the congruence subgroup property for $\Aut(F_2)$.
\end{abstract}
\section{Introduction}\label{S:I}

Let $G$ be a finitely generated group, $\Gamma = \Aut G$ be its
automorphism group. For a normal subgroup $K \subset G$ of finite
index, we set
$$
\Gamma[K] = \{ \sigma \in \Gamma \: \vert \: \sigma(K) = K \ \&  \
\sigma \ \text{acts trivially on} \ G/K \}.
$$
It is easy to see that $\Gamma[K]$ is a finite index subgroup of
$\Gamma.$ In fact, for $G = \Z^{\ell}$ we have $\Gamma =
GL_{\ell}(\Z),$ and furthermore if $K = n\Z^{\ell}$ then $\Gamma[K]
= GL_{\ell}(\Z , n),$ the congruence subgroup modulo $n.$ So, the
following question is a natural analog of the classical congruence
subgroup problem:

\vskip2mm

\hskip5mm {\it Does every finite index subgroup of $\Gamma$ contain
a suitable congruence subgroup $\Gamma[K]?$} \hfill (*)

\vskip2mm

\noindent While there are numerous results on the congruence
subgroup problem for arithmetic groups (cf. \cite{PR} for a recent
survey), very little is known regarding (*) for the automorphism
groups of general groups. The purpose of this note is to give a
short purely group-theoretic proof of Asada's result \cite{As} that
yields the congruence subgroup property (i.e., the the affirmative
answer to (*)) for $\Gamma = \Aut F_2,$ the automorphism group of
the free group $G = F_2$ of rank two. The original argument in
\cite{As} was based on the techniques involving Galois extensions of
rational function fields of algebraic curves (this area is generally
referred to as ``anabelian geometry"), and according to some
experts, no direct proof was known. The proof we present here is, by
and large, a ``translation" of Asada's argument into the
group-theoretic language. One of the benefits of the translation is
that some simplifications and shortcuts in Asada's argument became
apparent making the resulting argument very short and, in some
sense, even explicit (cf. \S \ref{S:E}). It also reveals the
underlying idea of the method (which we call the ``topsy-turvy
effect," see Remark 4.5) so that it can potentially be applied to
other automorphism groups and their subgroups.

Before formulating the result, we need to recall the standard
reformulation of (*) as a question about the comparison of two
topologies on $\Gamma.$ Let $\tau_{\mathrm{pf}}$ (resp.,
$\tau_{\mathrm{c}}$) be the topology on $\Gamma$ that admits the
family of all subgroups of finite index in $\Gamma$ (resp., the
family of congruence subgroups $\Gamma[K]$ for all finite index
subgroups $K \subset G$) as a fundamental system of neighborhoods of
the identity\footnote{We note that $\tau_{\mathrm{c}}$ can also be
defined using the congruence subgroups $\Gamma[K]$ associated only
to {\it characteristic} subgroups $K \subset G$ of finite index. For
such $K,$ $\Gamma[K] = \mathrm{Ker}\left(\Gamma \to
\Aut(G/K)\right),$ hence a {\it normal} subgroup of finite index in
$\Gamma.$}. Then $\tau_{\mathrm{c}}$ is {\it a priori} weaker than
$\tau_{\mathrm{pf}},$ and (*) amounts to the question if these
topologies are actually identical. Now, let $\widehat{\Gamma}$ and
$\overline{\Gamma}$ be the completions of $\Gamma$ relative to
$\tau_{\mathrm{pf}}$ and $\tau_{\mathrm{c}}$
respectively. Then yet another equivalent
reformulation of (*) is whether or not the natural map
$\widehat{\Gamma} \to \overline{\Gamma}$ is injective. Clearly,
$\widehat{\Gamma}$ is simply the profinite completion of $\Gamma.$
On the other hand, it is
easy to see that $\overline{\Gamma}$ can be identified with the
closure of the image of the natural homomorphism $\Aut G \to \Aut
\widehat{G}$ (cf. \S \ref{S:P}). So, our question becomes if the
natural map $\widehat{\Aut G} \to \Aut \widehat{G}$ is injective.
Here two remarks are in order. First, for a {\it profinite} group
$F$ we, of course, use $\Aut F$ to denote the group of {\it
continuous} automorphisms of $F;$ it is known however that if $F$ is
finitely generated then every abstract automorphism of $F$ is
automatically continuous \cite{NSe}. Second, for a finitely
generated profinite group $F,$ the automorphism group $\Aut F$ is
itself profinite (cf. \S \ref{S:P}), so the above homomorphism
$\widehat{\Aut G} \to \Aut \widehat{G}$ actually results from the
universal property of profinite completions applied to the
homomorphism $\Aut G \to \Aut \widehat{G}.$ Likewise, the outer
automorphism group $\Out \widehat{G} = \Aut \widehat{G} / \Int
\widehat{G}$ is also profinite, so the natural homomorphism $\Out G
\to \Out \widehat{G}$ extends to a continuous homomorphism
$\widehat{\Out G} \to \Out \widehat{G}.$ We can now formulate the
main result.

\vskip3mm

\noindent {\bf Main Theorem.} (cf. \cite{As}, Theorem 5) {\it For
the free group $F_2$ on two generators, the natural homomorphism
$\widehat{\Out F_2} \to \Out \widehat{F_2}$ is injective.}

\vskip3mm

The following is easily derived from the theorem (cf. Lemma
\ref{L:P1}).

\vskip3mm

\noindent {\bf Corollary.} {\it The natural homomorphism
$\widehat{\Aut F_2} \to \Aut \widehat{F_2}$ is injective, hence
$\Aut F_2$ has the congruence subgroup property.}

\vskip3mm

(We note that Asada's theorem was interpreted in \cite{AR}, 1.4.2,
as the statement that every finite index subgroup of $\Gamma =
\Aut F_2,$ containing $\Int F_2,$ must contain a suitable congruence
subgroup $\Gamma[K],$ but as we see, it in fact yields this property
for {\it all} finite index subgroups, cf. also Remark 5.3(3).)

\vskip3mm
The congruence subgroup property for $\mathrm{Aut}\: F_2$ can be used to establish 
the congruence subgroup property for certain subgroups (which are analogs 
of parabolic subgroups) of $\mathrm{Aut}\: F_n$ for $n\geqslant 3$ -- we will address this issue elsewhere. 
On the other hand, the congruence subgroup  problem for the group $\mathrm{Aut}\: F_n,$ $n \geqslant 3,$ 
itself remains widely open, and it does not appear that the argument for $n = 2$ can be easily extended to 
$n \geqslant 3.$ It is interesting that the proof for $n = 2$ relies on the fact that 
$\mathrm{Out}\: F_2 \simeq GL_2(\mathbb{Z})$ is a virtually free group, which is precisely  
what prevents $GL_2(\mathbb{Z})$ from having the (usual) congruence subgroup property. The latter
is a classical result known already to Klein and Fricke in the 19th century (cf. \cite{PR}).

\vskip3mm

The structure of the note is the following. In \S \ref{S:PF}, we
review the facts about profinite groups needed in the proof of the
Main Theorem. This section is included for the reader's convenience
as although these fact are known, for some of them it is not easy to
find an impeccable reference. More importantly, the proofs we
present, unlike the traditional proofs (cf., for example,
\cite{HRi}), are based not on the structure theory for profinite
groups but rather on the analysis of finite quotients of (discrete)
free groups and the associated relation modules. We will use this
approach to give an ``explicit" form of the Main Theorem in \S
\ref{S:E} (cf. Theorem \ref{T:E1}). After some reductions in \S
\ref{S:P}, we present the group-theoretic ``translation" of Asada's
argument in \S \ref{S:P2}. Finally, in \S \ref{S:T}, we discuss the
topological nature of a homomorphism involved in the proof of the
Main Theorem. This theme is prominent in Asada's paper, however the
explicit computation of this homomorphism in \S \ref{S:T} (as
opposed to its description in terms of Galois groups) reveals a
shortcut used in \S \ref{S:P2}.

\vskip4mm

\section{Facts about profinite groups}\label{S:PF}

We refer to \cite{RiZ} or \cite{W} regarding basic notions,
notations and results on profinite groups. In particular, the
profinite completion of an abstract (discrete) group $G$ will be
denoted by $\widehat{G}.$ Thus, if $F = F(X)$ is the free group on a
finite set $X$ then $\mathfrak{F} := \widehat{F}$ is the free
profinite group on $X$ (to be denoted $\mathfrak{F}(X)$). Given a
subset $S$ of a profinite group $\mathfrak{G},$ we will write
$\widehat{S}$ to denote its closure in $\mathfrak{G}.$
\begin{lemma}\label{L:PF1}
{\rm (cf. \cite{And}, Proposition 3)} Let
$$
1 \to G_1 \stackrel{\alpha}{\longrightarrow} G_2
\stackrel{\beta}{\longrightarrow} G_3 \to 1
$$
be an exact sequence of groups. Assume that $G_1$ is finitely
generated and that its profinite completion $\widehat{G}_1$ has
trivial center. Then the sequence of the profinite completions
$$
1 \to \widehat{G}_1 \stackrel{\widehat\alpha}{\longrightarrow}
\widehat{G}_2 \stackrel{\widehat\beta}{\longrightarrow}
\widehat{G}_3 \to 1
$$
is also exact.
\end{lemma}
\begin{proof}
We  identify $G_1$ with a normal subgroup of
$G_2$ and consider the conjugation action of the latter on the
former. This action extends to an action of $G_2$ on $\widehat{G}_1$
giving rise to a homomorphism $G_2 \to \mathrm{Aut}\:
\widehat{G}_1.$ On the other hand, since $G_1$ is finitely
generated, it is easy to see that the group $\mathrm{Aut}\:
\widehat{G}_1$ is profinite (cf. the beginning of \S3), so the
homomorphism $G_2 \to \mathrm{Aut}\: \widehat{G}_1$ extends to a
continuous homomorphism $\phi \colon \widehat{G}_2 \to
\mathrm{Aut}\: \widehat{G}_1$ such that $\phi(\hat{\alpha}(x)) =
\mathrm{Int}\: x$ for all $x \in \widehat{G}_1.$ Now, given $x \in
\widehat{G}_1,$ $x \neq 1,$ the assumption that $\widehat{G}_1$ has
trivial center implies that $\mathrm{Int}\: x$ is nontrivial, hence
$\hat{\alpha}(x) \neq 1,$ proving that $\hat{\alpha}$ is injective.

Since $\beta(G_2) = G_3,$ we have that $\widehat{\beta}(\widehat
G_2)$ is a compact dense subgroup of $\widehat G_3,$ yielding the
surjectivity of $\widehat{\beta}.$ Finally, $\beta$ defines an
isomorphism of $G_2 / \mathrm{Im}\: \alpha$ onto $G_3,$ and the
inverse of this isomorphism gives rise to a natural map $G_3 \to
\widehat{G_2}/\mathrm{Im}\: \widehat{\alpha}.$ It is easy to see
that the latter satisfies the universal property for the profinite
completion of $G_3,$ which yields $\mathrm{Im}\: \widehat{\alpha} =
\mathrm{Ker}\: \widehat{\beta},$ as required.
\end{proof}

\vskip.5mm

The proof of the Main Theorem relies on the known results about the
centralizers of generators and their commutators in free profinite
groups. As we already mentioned in \S \ref{S:I}, for the purpose of
giving an ``explicit" version of the Main Theorem (cf. Theorem
\ref{T:E1}), we present the proofs of these results based on the
analysis of finite quotients of free groups and their relation
modules rather than on the structure theory of profinite groups
(cf., for example, \cite{HRi}).
\begin{prop}\label{P:Pr10}
Let $F = F(X)$ be the free group on $X = \{x_1, \ldots , x_n\},$ and
let $G = F/N$ be a finite quotient of $F.$ Fix a prime $p$ not
dividing the order of $G,$  set $M = N^p[N , N]$ and let $\gamma
\colon F/M \to F/N$ denote the canonical homomorphism. Then

\vskip2mm

{\rm \ (i)} $\gamma(C_{F/M}(x_iM))$ coincides with the cyclic group
$\la x_iN \ra$ for all $i = 1, \ldots , n;$

\vskip2mm

{\rm (ii)} if $n > 1$ then for any abelian normal subgroup $C
\subset F/M$ we have $\gamma(C) = \{ e \}.$
\end{prop}
\begin{proof}
The proof uses some well-known properties of the relation module
$\mathfrak{n} := N/[N , N]$ (cf. \cite{Gr}). Namely, there is an
exact sequence of $\Z[G]$-modules
\begin{equation}\label{E:Pr10}
0 \to \mathfrak{n} \stackrel{\sigma}{\longrightarrow} \Z[G]^n
\stackrel{\tau}{\longrightarrow} \mathfrak{g} \to 0
\end{equation}
where $\mathfrak{g}$ is the augmentation ideal in $\Z[G].$ We recall
the construction of $\sigma$ and $\tau.$ It is known that the
augmentation ideal $\mathfrak{f} \subset \Z[F]$ is a free left
$\Z[F]$-module with basis $x_1 - 1, \ldots , x_n - 1.$ So, for any
$f \in F,$ there is a unique presentation of the form
$$
f - 1 = a_1(f)(x_1 - 1) + \cdots + a_n(f)(x_n - 1) \ \ \text{with} \
\ a_i(f) \in \Z[F],
$$
and then $\sigma$ is defined by sending $f \in N$ to
$(\overline{a_1(f)}, \ldots , \overline{a_n(f)}) \in \Z[G]^n,$
where the bar denotes the image under the natural homomorphism
$\Z[F] \to \Z[G].$ Furthermore, $\tau$ is defined by sending $(a_1,
\ldots , a_n) \in \Z[G]^n$ to $\sum a_i(\overline{x}_i - 1).$ Since
all terms in (\ref{E:Pr10}) are free $\Z$-modules, by tensoring with
$\F_p = \Z/p\Z,$ we obtain the following exact sequence of
$\F_p[G]$-modules
\begin{equation}\label{E:Pr11}
0 \to \mathfrak{n}_p \stackrel{\sigma_p}{\longrightarrow} \F_p[G]^n
\stackrel{\tau_p}{\longrightarrow} \mathfrak{g}_p \to 0
\end{equation}
where $\mathfrak{n}_p = \mathfrak{n} \otimes_{\Z} \F_p$ and
$\mathfrak{g}_p = \mathfrak{g} \otimes_{\Z} \F_p.$ Clearly,
$\mathfrak{n}_p = N/M,$ and $\mathfrak{g}_p$ is the augmentation
ideal in $\F_p[G].$ Since $p \nmid \vert G \vert,$ exact sequence
(\ref{E:Pr11}) splits, yielding an isomorphism of $\F_p[G]$-modules
\begin{equation}\label{E:Pr15}
N/M \simeq \F_p[G]^{n-1} \oplus \F_p,
\end{equation}
where $\F_p$ is considered as the trivial $\F_p[G]$-module. Besides,
since the order of $G$ is relatively prime to that of $N/M,$ there
exists a semi-direct product decomposition
$F/M \simeq N/M \rtimes G,$
which we fix (it is not canonical). This enables us to view $G$ as a
subgroup of $F/M.$

\vskip1mm

We will now prove assertion (i). To keep our notations simple, we
will write the argument for $i = 1.$ Let $g \in C_{F/M}(x_1M).$
Since $\la g \ra$ and $\la g^{p^{\ell}} \ra,$ for any $\ell
\geqslant 1,$ have the same image under $\gamma,$ it is enough to
prove our claim assuming that the order of $g$ is prime to $p.$ Then
there exists $h \in N/M$ such that $g' := hgh^{-1}$ belongs to $G.$
Let $d$ be the order of $\overline{x}_1$ in $G.$ Then $y:=x_1^dM \in
N/M,$ and, since $g$ commutes with $y$ and $N/M$ is commutative, $g'$ commutes with $y.$ In
other words, if $N/M$ is viewed as $\F_p[G]$-module then $y \in
N/M,$ and hence $\sigma_p(y),$ is fixed by $g'.$
Using the description of $\sigma$ given above, we obtain
$$
\sigma_p(y) = \left(\sum_{j = 0}^{d-1} \overline{x}_1^j, \ 0,\
\ldots ,\  0 \right) \ \in \F_p[G]^n.
$$
So, $g' \sum_{j = 0}^{d-1} \overline{x}_1^j = \sum_{j = 0}^{d-1}
\overline{x}_1^j$ in $\F_p[G],$ which implies that $\gamma(g) = g'
\in \la x_1N \ra,$ as required.

\vskip1mm

To prove (ii), we observe that for any $\ell \geqslant 1,$ the
subgroup $C^{p^{\ell}}$ is also normal in $F_n/M$ and
$\gamma(C^{p^{\ell}}) = \gamma(C).$ Since $C$ is abelian, for a
sufficiently large $\ell,$ the subgroup $C^{p^{\ell}}$ has order
prime to $p,$ and we can assume that $C$ has this property. Then there
exists $h \in N/M$ such that $hCh^{-1} \subset G,$
and since $C$ is normal, we actually have $C \subset G.$ If $C \neq
\{ e \}$ then as $n > 1,$ we conclude from (\ref{E:Pr15}) that there
exist $c \in C,$ $g \in N/M$ such that $cgc^{-1} \neq g.$ Then
$$
1 \neq cgc^{-1}g^{-1} \in C \cap N/M,
$$
a contradiction.
\end{proof}

\begin{cor}\label{C:Pr1}
Let $\mathfrak{F} = \mathfrak{F}(X)$ be the free profinite group on
a finite set $X.$ Then for any $x \in X,$ the centralizer
$C_{\mathfrak{F}}(x)$ coincides with the pro-cyclic group
$\widehat{\la x \ra}.$ Consequently, if $\vert X \vert > 1$ then
$\mathfrak{F}$ has trivial center.
\end{cor}
\begin{proof}
Let $F = F(X)$ be the free discrete group viewed as a (dense)
subgroup of $\mathfrak{F}.$ If our assertion is false then there
exists an open normal subgroup $U \subset \mathfrak{F}$ such that
the image of $C_{\mathfrak{F}}(x)$ in $G := \mathfrak{F}/U$ strictly
contains $\la xU \ra.$ Let $N = F \cap U.$ Pick a prime $p$ not
dividing $\vert G \vert,$ and let $V$ denote the closure of $M =
N^p[N , N]$ in $\mathfrak{F}$ (so that $F \cap V = M$). Then the
image of $C_{\mathfrak{F}}(x)$ in $\mathfrak{F}/V \simeq F/M$ is
contained in $C_{F/M}(xM).$ On the other hand, by Proposition
\ref{P:Pr10}(i), the image of $C_{F/M}(xM)$ in $F/N$ coincides with
$\la xN \ra.$ Using the natural isomorphism $F/N \simeq
\mathfrak{F}/U,$ we see that the image of $C_{\mathfrak{F}}(x)$ in
$\mathfrak{F}/U$ is contained in $\la xU \ra.$ A contradiction,
proving our first assertion. Since for $x , y \in X,$ $x \neq y,$ we
have $\widehat{\la x \ra} \cap \widehat{\la y \ra} = \{ e \},$ our
second assertion follows.
\end{proof}

\vskip1mm

\begin{cor}\label{C:Pr2}
{\rm (cf. \cite{As}, Lemma 10)} Let $\mathfrak{F} = \mathfrak{F}(X)$
be the free profinite group on a finite set $X$ with $\vert X \vert
> 1.$ If $C \subset \mathfrak{F}$ is an abelian normal subgroup then
$C = \{ e \}.$
\end{cor}
\begin{proof}
Again, let $F \subset \mathfrak{F}$ be the abstract free group
generated by $X.$ Assume that $C \neq \{ e \},$ and choose an open
normal subgroup $U \subset \mathfrak{F}$ that does not contain $C.$
As in the proof of Corollary \ref{C:Pr1}, set $N = F \cap U$ and $M
= N^p[N , N]$ where $p$ is a prime not dividing the order of $G :=
\mathfrak{F}/U.$ Let $V$ denote the closure of $M$ in
$\mathfrak{F}.$ Then $CV/V$ is an abelian normal subgroup of
$\mathfrak{F}/V = F/M.$ So, it follows from Proposition
\ref{P:Pr10}(ii) that its image in $F/N$ is trivial. Using the
isomorphism $F/N \simeq \mathfrak{F}/U,$ we conclude that $C \subset
U,$ a contradiction.
\end{proof}

\vskip1mm

\noindent {\bf Remark 2.5.} The proof of Lemma 10 in \cite{As} is
faulty. It is based on the ``fact" that if $\mathfrak{H}$ is a free
profinite group (of finite rank), and $h \in
\mathfrak{H}^{\small\mathrm{ab}} = \mathfrak{H}/[\mathfrak{H} ,
\mathfrak{H}]$ is a nontrivial element then for $\beta \in
\widehat{\Z},$ we have that $\beta h = 0$ in
$\mathfrak{H}^{\small\mathrm{ab}}$ implies that $\beta = 0.$ This is
false in general as $\widehat{\Z}$ has zero divisors. The argument
in \cite{As}, however, can be corrected by passing to the
corresponding free pro-$p$ group $\mathfrak{H}_p$ for some prime $p$
(for this, one needs an analog of Corollary \ref{C:Pr1} for
$\mathfrak{H}_p$).

\addtocounter{thm}{1}

\vskip2mm

We now recall Schreier's method (cf. \cite{MKS}, 2.3, or \cite{Rob},
Ch. VI), which will be used repeatedly in this note : Let $G$ be a
group with a generating set $X,$ let $H$ be a subgroup of $G,$ and
$T$ be a right transversal (i.e., a system of representatives of
right cosets containing the identity) to $H$ in $G.$ Given $g \in
G,$ we let $\overline{g}$ denote the unique element in $T$
satisfying $Hg = H\overline{g}.$ Then the set
$$
Y := \{ tx (\overline{tx})^{-1} \ \vert \ t \in T, \ x \in X \}
\setminus \{ e \},
$$
generates $H.$ Moreover, if $G$ is the free group on $X,$ and $T$
has the Schreier property (i.e., the initial segment of any element
of $T$ is again in $T$), then $Y$ is a free generating set for $H.$

\vskip3mm

\begin{lemma}\label{L:Pr25}
Let $F = F(X)$ be the free group on $X = \{x_1, \ldots , x_n\}$ with
$n > 1,$ and let $G = F/N$ be a finite quotient of $F.$ Pick a prime
$p$ not dividing $6\vert G \vert,$ and set $L = N \cap F^6[F , F]$
and $M = L^p[L , L].$ Let $\delta \colon F/M \to F/L$ be the
canonical homomorphism. Then for $c = [x_i , x_j] =
x_ix_jx_i^{-1}x_j^{-1},$ with $i \neq j,$ we have
$\delta(C_{F/M}(cM)) = \la cL \ra.$
\end{lemma}
\begin{proof}
Let $F_1 = F^2[F , F]$ and $F_2 = F^3[F , F].$ We will write the
argument for $i = 2,$ $j = 1.$ The set $$T_1 = \{ x_1^{e_1} \cdots
x_n^{e_n} \ \vert \ e_i \in \{ 0 , 1 \}\}$$ is a transversal to
$F_1$ in $F.$ Applying Schreier's method, we see that $F_1$ is a
free group on a set containing $c.$ Then by Proposition
\ref{P:Pr10}(i), we have $\delta(C_{F_1/M}(cM)) = \la cL \ra.$ Since
$F^2 \subset F_1,$ we obtain that
\begin{equation}\label{E:Pr30}
\text{for any} \ t \in C_{F/M}(cM)\  \text{we have} \ \delta(t)^2
\in \la cL \ra.
\end{equation}
Similarly,
$$
T_2 = \{ x_1^{e_1} \cdots x_n^{e_n} \ \vert \ e_i \in \{ 0, 1, 2
\}\}
$$
is a transversal to $F_2$ in $F.$ Applying again Schreier's method,
we obtain that $F_2$ is a free group on a set containing $c,$ hence
$\delta(C_{F_2/M}(cM)) = \la cL \ra.$ As $F^3 \subset F_2,$ we see
that
\begin{equation}\label{E:Pr31}
\text{for any} \ t \in C_{F/M}(cM)\  \text{we have} \ \delta(t)^3
\in \la cL \ra.
\end{equation}
Now, our assertion follows from (\ref{E:Pr30}) and (\ref{E:Pr31}).
\end{proof}

\begin{cor}\label{C:Pr30}
{\rm (cf. \cite{As}, Lemma 1)} Let $\mathfrak{F} = \mathfrak{F}(X)$
be the free profinite group on a finite set $X$ with $\vert X \vert
> 1.$ Given $x , y \in X,$ $x \neq y,$ for $c = [x, y]$ we have
$C_{\mathfrak{F}}(c) = \widehat{\la c \ra}.$
\end{cor}
This is derived from Lemma \ref{L:Pr25} just as Corollary
\ref{C:Pr1} was derived from Proposition \ref{P:Pr10}(i).

\vskip1mm

\begin{lemma}\label{L:PF4}
{\rm (cf. \cite{As}, Lemma 8)} Let $\mathfrak{F}$ be a free
profinite group of finite rank, and $N \subset \mathfrak{F}$ be an
open subgroup. If $\sigma \in \Aut \mathfrak{F}$ restricts trivially
to $N$ then $\sigma = \mathrm{id}_{\mathfrak{F}}.$
\end{lemma}
\begin{proof}
We can obviously assume that $N$ is normal in $\mathfrak{F},$ and
then  $g^m \in N$ for any $g \in \mathfrak{F},$ where $m =
[\mathfrak{F} : N].$ If $\mathfrak{F}$ is of rank one then $\sigma$
is of the form $\sigma(g) = g^{\alpha}$ for some $\alpha \in
\widehat{\Z}.$
The fact that $\sigma \vert N = \mathrm{id}_N$ implies that
$m(\alpha - 1) = 0.$ Since $m$ is not a zero divisor in
$\widehat{\Z},$ we conclude that $\alpha = 1,$ i.e. $\sigma =
\mathrm{id}_{\mathfrak{F}}.$

Now, assume that $\mathfrak{F} = \mathfrak{F}(X)$ where $\vert X
\vert > 1,$ and pick two distinct elements $x_1 , x_2 \in X.$ Since
$gx_i^mg^{-1} \in N$ for $i = 1, 2$ and all $g \in \mathfrak{F},$ we
have
$$
g x_i^m g^{-1}=\sigma(g x_i^m g^{-1})=\sigma(g) x_i^m
\sigma(g)^{-1},
$$
and therefore
\begin{equation}\label{E:Pr125}
g^{-1}\sigma(g)\in C_{\mathfrak{F}}(x_i^m) \ \ \text{for} \ i = 1,
2.
\end{equation}
We will now show that
\begin{equation}\label{E:Pr50}
\text{for any} \ x \in X \ \text{and any positive } \ m \in \Z, \
\text{we have} \ \ C_{\mathfrak{F}}(x^m) = \widehat{\la x \ra}.
\end{equation}
Indeed, consider the homomorphism $\varepsilon \colon F \to \Z/m\Z$
of the group $F = F(X)$ that takes $x$ to $1(\md m),$ and all other
generators $y \in X \setminus \{ x \}$ to $0(\md m).$ Let $H =
\mathrm{Ker}\: \varepsilon.$ Applying Schreier's method to the
transversal $T = \{x^e \ \vert \ e = 0, \ldots , m-1 \}$ to $H$ in
$F,$ we see that $H$ is a free group on a set containing $x^m.$ By
Corollary \ref{C:Pr1}, for the corresponding free profinite group
$\mathfrak{H}$ (the closure of $H$ in $\mathfrak{F}$) we have
$C_{\mathfrak{H}}(x^m) = \widehat{\la x^m \ra}.$ Since
$[\mathfrak{F} : \mathfrak{H}] = m,$ we have $[C_{\mathfrak{F}}(x^m)
: C_{\mathfrak{H}}(x^m)] \leqslant m.$ On the other hand
$\widehat{\la x \ra} \subset C_{\mathfrak{F}}(x^m)$ and
$[\widehat{\la x \ra} : \widehat{\la x^m \ra}] = m,$ so
(\ref{E:Pr50}) follows.

Using (\ref{E:Pr125}) and (\ref{E:Pr50}), we now see that
$$
g^{-1}\sigma(g) \in \widehat{\la x_1 \ra} \cap \widehat{\la x_2 \ra}
= \{ e \} \ \ \text{for any} \ \ g \in \mathfrak{F},
$$
i.e. $\sigma = \mathrm{id}_{\mathfrak{F}}.$
\end{proof}

\section{Some reductions}\label{S:P}

Let $\mathfrak{F}$ be a finitely generated profinite group. Then
$\mathfrak{F}$ has only finitely many open subgroups of index $\leqslant n,$ 
for each $n \geqslant 1,$ implying that the intersection
$U_n$ of all these subgroups is itself an open normal subgroup. It
is easy to see that the automorphism group $\Aut \mathfrak{F}$ can
be naturally identified with $\displaystyle \lim_{\leftarrow}
\Aut(\mathfrak{F}/U_n),$ making it a profinite group. Furthermore,
the topology on $\Aut \mathfrak{F}$ arising from the above
identification coincides with the natural topology of uniform
convergence (cf. \cite{Sm}). Now, if $\mathfrak{F} = \widehat{G},$
where $G$ is a finitely generated discrete group, then the pullback
of the topology on $\Aut \widehat{G}$ under the natural map $\Aut G
\stackrel{\iota}{\longrightarrow} \Aut \widehat{G}$ coincides with
the topology $\tau_{\mathrm{c}}$ defined in terms of congruence
subgroups $\Gamma[K]$ of $\Gamma = \Aut G$ for all finite index
subgroups $K \subset G$ (cf. \S \ref{S:I}). Consequently, the
completion $\overline{\Gamma}$ can be identified with the closure of
$\mathrm{Im}\: \iota$ in $\Aut \widehat{G}.$
The kernel of the resulting map
$\widehat{\Aut G}\stackrel{\varphi}{\longrightarrow} \Aut \widehat{G}$
coincides with
the intersection $\bigcap_K \widehat{\Gamma[K]}$ taken over all
finite index normal subgroups $K \subset G$ where $\widehat{\ }$
denotes the closure in $\widehat{\Gamma} = \widehat{\Aut
G}.$\footnote{We note that for $G$ a free group of any rank $r
\geqslant 1,$ the homomorphism $\widehat{\Aut G} \to \Aut
\widehat{G}$ is not surjective. This follows from the fact that
$\widehat{G}^{\small\mathrm{ab}} = \widehat{\Z}^{r},$ and the
resulting map $\Aut \widehat{G}
\stackrel{\hat{\theta}}{\longrightarrow} GL_{r}(\widehat{\Z})$ is
surjective while $(\hat{\theta} \circ \varphi)(\widehat{\Aut G})$ is
contained in (actually, is equal to) the subgroup of
$GL_{r}(\widehat{\Z})$ of matrices having determinant $\pm 1.$
Incidentally, it is well-known (and follows, for example, from
Dirichlet's Prime Number Theorem) that $\widehat{\Z}^{\times}$ is
not finitely generated, implying that $\Aut \widehat{G}$ is not
finitely generated (as a profinite group) - this result is given as
Corollary 3 in \cite{Ro} where it is established using some results
of \cite{Me}.} Similarly, the group of outer automorphisms  $\Out
\widehat{G} = \Aut \widehat{G} / \Int \widehat{G}$ is profinite, so
the natural map $\Out G \stackrel{\omega}{\longrightarrow} \Out
\widehat{G}$ extends to a continuous homomorphism $\widehat{\Out G}
\stackrel{\psi}{\longrightarrow} \Out \widehat{G}.$ As above, the
pullback under $\omega$ of the topology on $\Out \widehat{G}$
coincides with  the topology on $\Delta = \Out G$ defined by the
following ``congruence subgroups"
$$
\Delta[K] = \mathrm{Ker}\left(\Out G \to \Out(G/K) \right)
$$
associated to characteristic finite index subgroups $K \subset G.$
Then the completion $\overline{\Delta}$ of $\Delta$ for that
topology can be identified with the closure of $\mathrm{Im}\:
\omega$ in $\Out \widehat{G},$ and $\mathrm{Ker}\: \psi$ coincides
with the intersection $\bigcap_K \widehat{\Delta[K]}$ where the
intersection is taken over all finite index characteristic subgroups
$K \subset G$ and $\widehat{\ }$ denotes the closure in the
profinite completion $\widehat{\Delta}.$

\vskip2mm

We will now relate the injectivity of $\varphi$ to that of $\psi.$
\begin{lemma}\label{L:P1}
Let $G$ be a finitely generated residually finite group such that
$\widehat{G}$ has trivial center. If $\widehat{\Out G}
\stackrel{\psi}{\longrightarrow} \Out \widehat{G}$ is injective then
so is $\widehat{\Aut G} \stackrel{\varphi}{\longrightarrow} \Aut
\widehat{G}.$
\end{lemma}
\begin{proof}
Since $G$ is residually finite, the center of $G$ is also trivial,
so identifying $\Int G$ and $\Int \widehat{G}$ with $G$ and
$\widehat{G}$ respectively, we get the following commutative diagram
with exact rows:
$$
\begin{array}{cccccccccc}
1 & \to & G & \longrightarrow & \Aut G & \longrightarrow & \Out G &
\to & 1 \\
 & &\downarrow & &\downarrow & &\downarrow & & & \\
1 & \to & \widehat{G} & \longrightarrow & \Aut \widehat{G} &
\longrightarrow & \Out \widehat{G} & \to & 1
\end{array}
$$
Since the center of $\widehat{G}$ is trivial, taking the profinite
completion of the groups in the top row yields, by Lemma
\ref{L:PF1}, the following commutative diagram with exact rows:
$$
\begin{array}{cccccccccc}
1 & \to &\widehat{G} & \longrightarrow & \widehat{\Aut G} &
\longrightarrow & \widehat{\Out G} &
\to & 1 \\
 & &\parallel \downarrow & &\varphi \downarrow & &\psi \downarrow & & & \\
1 & \to & \widehat{G} & \longrightarrow & \Aut \widehat{G} &
\longrightarrow & \Out \widehat{G} & \to & 1
\end{array}
$$
Then a simple diagram chase shows that if $\psi$ is injective then
$\varphi$ is also injective.
\end{proof}

\vskip5mm

Let $F$ be the free group with generators $x$ and $y.$ It is
well-known (cf. \cite{MKS}, 3.5) that the canonical homomorphism $F
\longrightarrow F^{\small\mathrm{ab}} = F/[F , F]$ combined with the
identification $F^{\small\mathrm{ab}} \simeq \Z^2$ yields the
following exact sequence:
$$
1 \to \In F \longrightarrow \Aut F
\stackrel{\theta}{\longrightarrow} \Aut F^{\small\mathrm{ab}} =
GL_2(\Z) \to 1,
$$
i.e., $\Out F$ can be naturally identified with $GL_2(\Z).$ Let
$\Phi$ be the free group with generators $a$ and $b.$ Consider the
automorphisms $\alpha , \beta \in \Aut F$ defined by
$$
\alpha \colon \left\{\begin{array}{ccc} x & \to & x \\ y & \to &
yx^2
\end{array} \right. \ \ \ \text{and} \ \ \ \beta \colon \left\{\begin{array}{ccc} x & \to & xy^2 \\ y & \to &
y
\end{array} \right.,
$$
and let $\Phi \to \Aut F$ be the homomorphism defined by $a \mapsto
\alpha$ and $b \mapsto \beta.$ Since the group $\Aut \widehat{F}$ is
profinite, this homomorphism extends to a continuous homomorphism
$\nu \colon \widehat{\Phi} \to \Aut \widehat{F}.$
\begin{prop}\label{P:P1}
If $\nu$ is injective then $\varphi: \widehat{\Aut F}\to \Aut \widehat{F}$ is also injective.
\end{prop}
\begin{proof}
According to Lemma \ref{L:P1}, it is enough to show that $\psi$ is
injective.  We will first establish the injectivity of the composite
map
$$
\lambda: \widehat{\Phi} \stackrel{\nu}{\longrightarrow} \Aut \widehat{F}
\longrightarrow \Out \widehat{F}.
$$
It is easy to see that $\alpha$ and $\beta$ fix $c = [x , y] =
xyx^{-1}y^{-1},$ so it follows from Corollary \ref{C:Pr30} that $D
:= \nu(\widehat{\Phi}) \cap \In \widehat{F}$ is contained in
$\widehat{\langle \In c \rangle}.$ Since $\nu$ is injective, we
conclude that $C : = \mathrm{Ker}\: \lambda = \nu^{-1}(D)$ is a
procyclic, hence abelian, normal subgroup of $\widehat{\Phi}.$
Applying Corollary \ref{C:Pr2}, we obtain that $C = \{ e \},$ i.e.
$\lambda$ is injective.

As we explained in the beginning of this section, $\mathrm{Ker}\:
\psi$ is contained in the closure $\widehat{\Delta[K]}$ for any
finite index characteristic subgroup $K \subset F.$ So, the
injectivity of $\psi$ will follow from that of $\lambda$ if we
establish the inclusion
\begin{equation}\label{E:P1}
\mathrm{Im}\: \mu \supset \Delta[K_0] \ \ \text{for} \ \ K_0 = F^4[F
, F],
\end{equation}
where $\mu$ denotes the composite map $\Phi \to \Aut F \to \Out F.$
However, under the identification $\Out F \simeq GL_2(\Z),$ the
subgroup $\Delta[K_0]$ corresponds to the congruence subgroup
$GL_2(\Z , 4) = SL_2(\Z , 4)$ modulo 4, and $\mathrm{Im}\: \mu$
corresponds to the subgroup $H \subset SL_2(\Z)$ generated by the
matrices
$$
\left(\begin{array}{cc} 1 & 2 \\ 0 & 1\end{array} \right) \ \
\text{and} \ \ \left(\begin{array}{cc} 1 & 0 \\ 2 & 1 \end{array}
\right).
$$
So, (\ref{E:P1}) follows from the well-known fact that $H$ contains
$SL_2(\Z , 4)$ (cf. \cite{Sa}).
\end{proof}

The following completes the proof of the Main Theorem.
\begin{prop}\label{P:P2}
The map $\nu$ is injective.
\end{prop}

\section{Proof of Proposition \ref{P:P2}}\label{S:P2}

We begin this section with some constructions needed in the proof of
Proposition \ref{P:P2}. Consider the following free product
$$
\Psi = \langle z_1 \rangle * \langle z_2 \rangle
* \langle z_3 \rangle \ \ \text{where} \ \  z_i^2 = e \ \ \text{for}\ \ i =1, 2, 3,
$$
and let $\varepsilon \colon \Psi \to \langle z_1 \rangle \times
\langle z_2 \rangle \times \langle z_3 \rangle$ be the canonical
homomorphism (which actually coincides with the homomorphism of
abelianization $\Psi \to \Psi^{\small\mathrm{ab}}$). It follows from
Kurosh's Theorem and the exercise in \cite{S}, Ch. I, \S 5.5, that
$\mathrm{Ker}\: \varepsilon$ is a free group of rank five. More
generally, any subgroup of $\Psi$ which does not meet any conjugate
of any of the factors is free. This, in particular, applies to the
kernel $\Theta$ of the following composite homomorphism
$$
\Psi \stackrel{\varepsilon}{\longrightarrow} \langle z_1 \rangle
\times \langle z_2 \rangle \times \langle z_3 \rangle
\stackrel{\sigma}{\longrightarrow} \Z/2\Z
$$
where $\sigma$ sends each $z_i$ to $1(\md 2)$. Choosing $\{1 , z_1\}$ as a transversal
to $\Theta$ in $\Psi$ and applying Schreier's method, we see that
$\Theta$ is generated by $z_2z_1, z_3z_1, z_1z_2$ and $z_1z_3,$ and
hence by $z_1z_2$ and $z_2z_3.$ Since $\Theta$ is obviously
nonabelian, it is the free group on $z_1z_2$ and $z_2z_3.$ We now
identify $F$ with $\Theta \subset \Psi$ using the (fixed) embedding
$F \hookrightarrow \Psi$ defined by $x \mapsto z_1z_2$ and $y
\mapsto z_2z_3.$
\begin{lemma}\label{L:Pr1}
There exist automorphisms $\da , \db \in \Aut \Psi$ such that
\begin{equation}\label{E:Pr1}
\da \colon \left\{\begin{array}{ccc} z_1 & \mapsto & (z_1z_2)^{-1}
z_1 (z_1z_2) \\ z_2 & \mapsto & (z_1z_2)^{-1} z_2 (z_1z_2) \\ z_3 &
\mapsto & (z_1z_2)^{-2} z_3 (z_1z_2)^2 \end{array} \right. \ \
\text{and} \ \ \db \colon \left\{\begin{array}{ccc} z_1 & \mapsto &
z_1
\\ z_2 & \mapsto & (z_2z_3)^{-1} z_2 (z_2z_3) \\ z_3 & \mapsto &
(z_2z_3)^{-1} z_3 (z_2z_3) \end{array} \right.
\end{equation}
Furthermore, $F$ is invariant under $\da$ and $\db$ and
$$
\da \: \vert\: F = \alpha \ \ \text{and} \ \ \db \: \vert\: F =
\beta.
$$
\end{lemma}
\begin{proof}
Let $\mathcal{F}$ be the free group on $\tilde{z}_1, \tilde{z}_2,
\tilde{z}_3.$ Then $\Psi = \mathcal{F}/\mathcal{N}$ where
$\mathcal{N}$ is the normal subgroup of $\mathcal{F}$ generated by
by $\tilde{z}_1^2, \tilde{z}_2^2, \tilde{z}_3^2.$ Let
$\tilde{\alpha} \colon \mathcal{F} \to \mathcal{F}$ be the
endomorphism of $\mathcal{F}$ defined by the replicas of equations
(\ref{E:Pr1}) written in terms of $\tz_1, \tz_2, \tz_3,$ i.e. $\tz_1
\mapsto (\tz_1\tz_2)^{-1} \tz_1 (\tz_1\tz_2)$ etc. Using the fact
that $\tilde{\alpha}(\tz_1\tz_2) = \tz_1\tz_2,$ it is easy to see
that $\mathrm{Im}\: \tilde{\alpha}$ contains $\tz_1, \tz_2$ and
$\tz_3,$ making $\tilde{\alpha}$ surjective. Since $\mathcal{F}$ is
hopfian (cf. \cite{Rob}, 6.1.12), we conclude that $\tilde{\alpha}$
is an automorphism of $\mathcal{F}$ (which can also be checked
directly). Clearly, $\tilde{\alpha}(\tz_1^2),
\tilde{\alpha}(\tz_2^2)$ and $\tilde{\alpha}(\tz_3^2)$ are contained
in $\mathcal{N},$ and in fact generate it as a normal subgroup of
$\mathcal{F}.$ Thus, $\tilde{\alpha}(\mathcal{N}) = \mathcal{N},$
and therefore
$\tilde{\alpha}$ descends to an automorphism $\da$ of $\Psi.$ By
direct computation we obtain that
$$
\da(x) = \da(z_1z_2) = z_1z_2 = x = \alpha(x)
$$
and
$$
\da(y) = \da(z_2z_3) = (z_2z_3)(z_1z_2)^2 = yx^2 = \alpha(y).
$$
Thus, $\da$ leaves $F$ invariant and restricts to $\alpha.$

\vskip1mm

The computation for $\db$ is similar (and even simpler). Again, we
observe that for the corresponding $\tilde{\beta}$ we have
$\tilde{\beta}(\tz_2\tz_3) = \tz_2\tz_3,$ using which one easily
verifies that $\tilde{\beta}$ is surjective, hence an automorphism
of $\mathcal{F}.$ Furthermore, $\tilde{\beta}(\mathcal{N}) =
\mathcal{N},$ so $\tilde{\beta}$ descends to $\db \in \Aut \Psi.$ We
have
$$
\db(x) = \db(z_1z_2) = (z_1z_2)(z_2z_3)^2 = xy^2 = \beta(x),
$$
and
$$
\db(y) = \db(z_2z_3) = z_2z_3 = y = \beta(y)
$$
which means that $\db$ also leaves $F$ invariant and restricts to
$\beta.$
\end{proof}

\vskip1mm

In the sequel, we will work with the homomorphism $\Phi \to \Aut
\Psi$ defined by $a \mapsto \da,$ $b \mapsto \db.$ Now, let
$$
F' = \{ x \in \Psi \ \vert \ \varepsilon(x) = (\epsilon_1 ,
\epsilon_2 , \epsilon_3) \in (\Z/2\Z)^3  \ \text{with} \ \
\epsilon_1 = \epsilon_2 = \epsilon_3 \}.
$$

\vskip1mm

\begin{lemma}\label{E:Pr2}
$F'$ is the free group on
$$
u = z_1z_2z_3 , \ \ v = z_2z_3z_1 \ \ \text{and} \ \ w = z_3z_1z_2.
$$
\end{lemma}
\begin{proof}
Clearly, $F'$ intersects trivially every conjugate of each factor
$\langle z_i \rangle,$ so it follows from Kurosh's theorem that $F'$
is a free group. As $\mathrm{Ker}\: \varepsilon$ is a free group of
rank 5 and $[F' : \mathrm{Ker}\: \varepsilon] = 2$ we conclude from
Schreier's formula that $F'$ is of rank 3. Taking $T = \{1, z_1,
z_2, z_3\}$ as a transversal to $F'$ in $\Psi$ and applying
Schreier's method, we see that $F'$ is generated by the following
set
$$
\{ z_1z_2z_3,\ z_1z_3z_2,\ z_2z_1z_3,\ z_2z_3z_1,\ z_3z_1z_2,\
z_3z_2z_1 \}.
$$
But $z_3z_1z_2 = (z_2z_1z_3)^{-1}$ etc, so $F'$ is generated by $u,
v$ and $w.$ Since $F'$ is hopfian, we conclude that it is the
free group on $\{u, v, w\}.$
\end{proof}

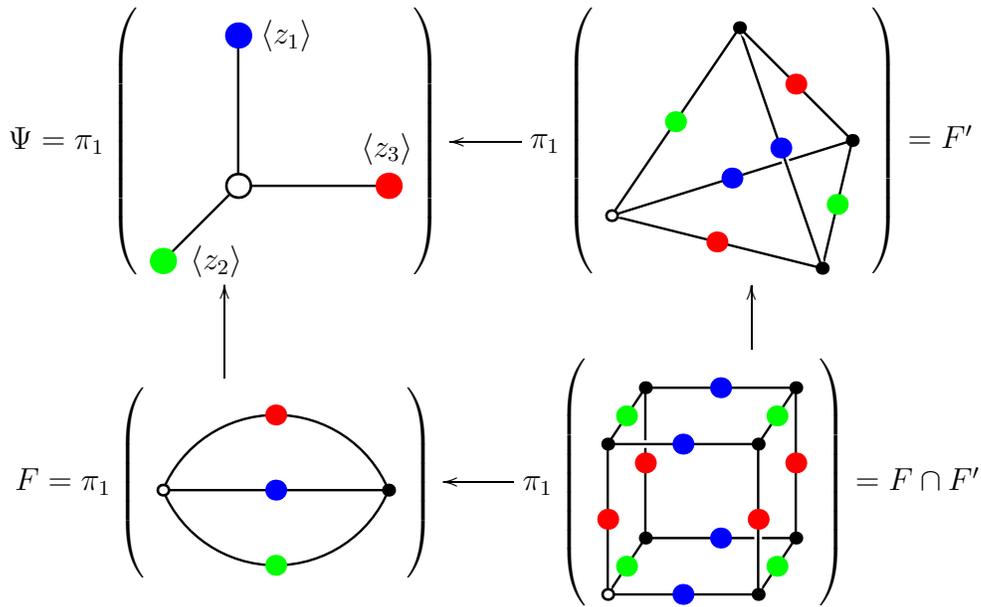
\begin{table}[h]
\[
  \xymatrix{
    {\Psi=\pi_1\left(%
      \cput{2.8cm}{ 1.4cm}{ \langle z_1 \rangle }%
      \cput{1.4cm}{-1.6cm}{ \langle z_2 \rangle }%
      \cput{4.8cm}{-0.1cm}{ \langle z_3 \rangle }%
      \insertgraph{10}%
    \right)}
    &
    {\pi_1\left(\insertgraph{21}\right)= F'}
    \ar[l]
    \\
    {F = \pi_1\left(\insertgraph{22}\right)}
    \ar[u]
    &
    {\pi_1\left(\insertgraph{31}\right) = F \cap F'}
    \ar[l]
    \ar[u]
  }
\]
\parbox[c]{\textwidth}{%
The groups $\Psi$, $F$, $F'$, and $F\cap F'$ can be realized
as the fundamental groups of suitable graphs of groups as shown above. In the case of $\Psi$
the vertex groups at the three terminal vertices are taken to be $\Z/2\Z$,
and all other vertex and edge groups are trivial. The inclusions correspond to
covering maps of graphs.%
}
\end{table}

\begin{Remark}\rm
Since $\da , \db \in \Aut \Psi$ act trivially on
$\Psi^{\small\mathrm{ab}},$ the group $\Phi$ leaves $F'$ invariant.
We note that the profinite completion $\widehat{F'}$ corresponds to
the Galois group $\Ga(\mathcal{M}_t/\mathcal{K})$ in \cite{As},
where it is asserted only that the latter is invariant under a
certain subgroup of index two $\Ga(\bar{k}/k_1) \subset
\Ga(\bar{k}/k)$ (cf. the end of \S 5 in \cite{As}). As a result, the
argument in \cite{As} involves (in our notations) an index two
subgroup $\widehat{\Phi}_1 \subset \widehat{\Phi}$ that leaves
$\widehat{F'}$ invariant, and amounts to proving first that the
restriction $\nu \: \vert \: \widehat{\Phi}_1$ is injective, and
then deriving that $\nu$ itself is injective. As $F'$ is in fact
$\Phi$-invariant, the step involving the introduction of
$\widehat{\Phi}_1$ can be eliminated from the argument.
\end{Remark}

The action of $\da , \db$ on $F'$ is described explicitly in the
following lemma.
\begin{lemma}\label{L:Pr3}
We have
$$
\da(u) = w^{-1} u w, \ \ \da(v) = v, \ \ \da(w) = (uw)^{-1} w (uw),
$$
and
$$
\db(u) = u, \ \ \db(v) = v, \ \ \db(w) = (vu)^{-1} w (vu).
$$
\end{lemma}
\begin{proof}
This is verified by direct computation. We have
$$
\da(u) = \da(z_1z_2z_3) = (z_2z_1z_3)(z_1z_2z_3)(z_3z_1z_2) = w^{-1}
u w,
$$
$$
\da(v) = \da(z_2z_3z_1) = z_2z_3z_1 = v,
$$
and
$$
\da(w) = \da(z_3z_1z_2) = (z_1z_2)^{-2}(z_3z_1z_2)(z_1z_2)^2 =
(uw)^{-1} w (uw).
$$
The computation for $\db$ is even easier:
$$
\db(u) = \db(z_1z_2z_3) = z_1z_2z_3 = u,
$$
$$
\db(v) = \db(z_2z_3z_1) = z_2z_3z_1 = v,
$$
and
$$
\db(w) = \db(z_3z_1z_2) = (z_2z_3)^{-2}(z_3z_1z_2)(z_2z_3)^2 =
(vu)^{-1} w (vu).
$$
\end{proof}

\vskip1mm

In the sequel, the restrictions of $\da , \db$ to $F'$ will be
denoted by $\alpha'$ and $\beta',$ respectively.

\vskip2mm

{\it Proof of Proposition \ref{P:P2}.} For a discrete (resp.,
profinite) group $G$ and its abstract (resp., closed) subgroup $H,$
we let $\aut(G , H)$ denote the subgroup of $\Aut G$ consisting of
those automorphisms that leave $H$ invariant. Since $\da , \db$ do
leave $F , F' \subset \Psi$ invariant, the homomorphism $\Phi \to
\Aut \Psi$ (given by $a\mapsto \da, b\mapsto \db$) leads to the following
commutative diagram in which all maps are given by restriction:
\begin{equation}\label{E:Pr25}
\begin{array}{ccc}
\Phi & \stackrel{\nu_0}{\longrightarrow} & \aut(F , F \cap F')  \\
\varkappa_0 \downarrow & & \downarrow \\
\aut(F' , F \cap F') & \longrightarrow & \Aut F \cap F'
\end{array}.
\end{equation}
Moreover, $\nu_0$ and $\varkappa_0$ send the generators $a , b$ of
$\Phi$ to $\alpha , \beta$ and $\alpha' , \beta'$ respectively (cf.
Lemma \ref{L:Pr1}). Then (\ref{E:Pr25}) gives rise to the following
commutative diagram
\begin{equation}\label{E:Pr55}
\begin{array}{ccc}
\widehat{\Phi} & \stackrel{\nu}{\longrightarrow} & \aut(\widehat{F} , \widehat{F \cap F'})  \\
\varkappa \downarrow & & \downarrow \\
\aut(\widehat{F'} , \widehat{F \cap F'}) & \longrightarrow & \Aut
\widehat{F \cap F'}
\end{array}.
\end{equation}
Assume that $\varkappa$ is injective. Then, given $x \in
\mathrm{Ker}\: \nu,$ we see from (\ref{E:Pr55}) that $\varkappa(x)
\in \Aut \widehat{F'}$ restricts trivially to $\widehat{F \cap F'}.$
Invoking Lemma \ref{L:PF4}, we see that $\varkappa(x) = 1,$ and
hence $x = 1.$

To prove the injectivity of $\varkappa,$ we consider the canonical
homomorphism $F' \to F'/V =: \bar{F}$ where $V$ is the normal
subgroup of $F'$ generated by $v.$ Clearly, $\bar{F}$ is the free
group on the images $\bar{u} , \bar{w}$ of $u$ and $w,$
respectively. The description of $\alpha' , \beta'$ given in Lemma
\ref{L:Pr3} implies that $\mathrm{Im}\: \varkappa_0$ is contained in
the subgroup $\aut(F' , v) \subset \Aut F'$ of all automorphisms
that fix $v.$ Let $\widehat{\bar{F}} = \widehat{F'}/\widehat{V}$
(where $\widehat{V}$ is the closure of $V$ in $\widehat{F'}$) be the
profinite completion of $\bar{F}$ and let $\aut(\widehat{F'} , v)
\subset \Aut \widehat{F'}$ be the subgroup of all automorphisms that
fix $v.$ We then have the following commutative diagram:
$$
\begin{array}{ccccc}
\Phi & \stackrel{\varkappa_0}{\longrightarrow} & \aut(F' , v) &
\longrightarrow & \Aut \bar{F} \\
\downarrow & & \downarrow & & \downarrow \\
\widehat{\Phi} & \stackrel{\varkappa}{\longrightarrow} &
\aut(\widehat{F'} , v) & \longrightarrow & \Aut \widehat{\bar{F}}
\end{array}
$$
It follows from Lemma \ref{L:Pr3} that the images of $a , b$ in
$\Aut \bar{F},$ and hence in $\Aut \widehat{\bar{F}},$ coincide with
the inner automorphisms $\Int \bar{u}\bar{w}$ and $\Int \bar{u},$
respectively. Since $\bar{u}\bar{w}$ and $\bar{u}$ freely generate
$\bar{F}$ and $\widehat{\bar{F}}$ has trivial center (Corollary
\ref{C:Pr1}), we conclude that the composite homomorphism
$\widehat{\Phi} \to \Aut \widehat{\bar{F}}$ is injective, and hence
$\varkappa$ is injective, as required. \hfill $\Box$

\begin{Remark}\rm The proof of Proposition \ref{P:P2} can
be informally, but adequately, described as the ``topsy-turvy
effect," in the following sense. Let us think about $\Out F$ and
$\Int F$ as the ``top" and the ``bottom" parts of $\Aut F.$ The
image of the homomorphism $\nu_0 \colon \Phi \to \Aut F$ that leads
to $\nu,$ has trivial intersection with $\Int F,$ so $\nu_0$ can be
characterized as a homomorphism to the top part of $\Aut F.$ In
essence, the proof of Proposition \ref{P:P2} is based on
constructing another free group on two generators $\bar{F},$ which
is a quotient of a group $F'$ commensurable with $F,$ and relating
$\nu_0$ to a homomorphism $\Phi \to \Aut \bar{F}$ whose image lies
in the bottom part $\Int \bar{F}.$ Then the injectivity of the
corresponding homomorphism $\widehat{\Phi} \to \Aut
\widehat{\bar{F}}$ reduces to the fact that $\widehat{\bar{F}}$ has
trivial center.
\end{Remark}

\section{Explicit construction}\label{S:E}

In this section, we will recast the proof of the Main Theorem in a
way that involves only finite quotients of free groups rather than 
free profinite groups. This leads to an
explicit procedure enabling one to construct, for a given finite
index normal subgroup $N$ of $\Gamma = \Aut F$ containing $\Int F$
(where $F$ is, as above, the free group on two generators, $x$ and
$y$), a finite index normal subgroup $K$ of $F$ such that $\Gamma[K]
\subset N.$

Let $SL'_2(\Z)$ denote the subgroup of $SL_2(\Z)$ (freely) generated
by $\left(\begin{array}{cc} 1 & 2 \\ 0 & 1 \end{array} \right)$ and
$\left(\begin{array}{cc} 1 & 0 \\ 2 & 1 \end{array} \right)$ (we
recall that $[SL_2(\Z) : SL'_2(\Z)] = 12$), and let $\mathrm{Aut}'\:
F$ be the preimage of $SL'_2(\Z)$ under the canonical homomorphism
$\Aut F \stackrel{\theta}{\longrightarrow} GL_2(\Z).$ Then replacing
$N$ with the normal subgroup $N \cap \mathrm{Aut}'\: F$ which
contains $\Int F$ and whose index in $\Gamma$ divides $12[\Gamma :
N],$ we can assume that $\Int F \subset N \subset \mathrm{Aut}'\:
F.$
\begin{thm}\label{T:E1}
Let $N$ be a finite index normal subgroup of $\Gamma$ such that
$\Int F \subset N \subset \mathrm{Aut}'\: F,$ and let $n =
[\mathrm{Aut}'\: F : N].$ Pick two distinct odd primes $p , q \nmid
n,$ and set $m = n \cdot p^{n+1}.$ Then there exist an explicitly
constructed normal subgroup $K \subset F$ of index dividing $144m^4
\cdot q^{36m^4 + 1}$ such that $\Gamma[K] \subset N.$
\end{thm}
\begin{proof}
We will freely use the notations introduced in the previous
sections; in particular, $\nu_0 \colon \Phi \to \Aut F$ will denote
the homomorphism given by $a \mapsto \alpha,$ $b \mapsto \beta.$ We
notice that $\theta \circ \nu_0$ is an isomorphism between $\Phi$
and $SL'_2(\Z);$ in particular, $\nu_0$ is injective and
\begin{equation}\label{E:E5}
\mathrm{Aut}'\: F = \Int F \rtimes \nu_0(\Phi).
\end{equation}
We let $\pi \colon \mathrm{Aut}'\: F \to \nu_0(\Phi) \to \Phi$
denote the homomorphism induced by the corresponding projection. It
follows from (\ref{E:E5}) that $\nu_0$ induces an isomorphism
$\bar{\nu}_0$ between $\Phi$ and $\mathrm{Out}'\: F :=
\mathrm{Aut}'\: F / \Int F.$ Let $\bar{N}$ denote the image of $N$
in $\mathrm{Out}'\: F,$ and let $\mathcal{N} =
\bar{\nu}_0^{-1}(\bar{N}).$ Since $p$ does not divide $n = \vert
\Phi/\mathcal{N}\vert,$ it follows from Proposition \ref{P:Pr10}(ii)
that $\mathcal{M} := \mathcal{N}^p[\mathcal{N} , \mathcal{N}]$ has
the following property:
\begin{equation}\label{E:E1}
\text{any cyclic normal subgroup of}\ \Phi/\mathcal{M} \ \text{has
trivial image in} \ \Phi/\mathcal{N}.
\end{equation}
Furthermore, according to (\ref{E:Pr15}), we have an isomorphism
$\mathcal{N}/\mathcal{M} \simeq \F_p[\Phi/\mathcal{N}] \oplus \F_p,$
which shows that $\vert \Phi/\mathcal{M} \vert$ equals $m = n \cdot
p^{n+1}.$ Observing that for an $(\mathrm{Aut}'\: F)$-invariant
subgroup $K \subset F,$ the congruence subgroup $\Gamma[K]$ is
normalized by $\mathrm{Aut}'\: F,$ we conclude from (\ref{E:E1})
that it is enough to explicitly construct an $(\mathrm{Aut}'\:
F)$-invariant subgroup $K \subset F$ of index dividing $144m^4 \cdot
q^{36m^4 + 1}$ such that
\begin{equation}\label{E:E2}
\Gamma[K] \subset \mathrm{Aut}'\: F \ \ \text{and\  the image of} \
\pi(\Gamma[K]) \ \text{in} \ \Phi/\mathcal{M} \ \text{is cyclic.}
\end{equation}
(Indeed, then the image of $\Gamma[K]$ in $\mathrm{Out}'\: F$ is
contained in $\bar{N},$ hence $\Gamma[K] \subset N$ as $N \supset
\Int F.$)

Now, let $\mathcal{L} = \mathcal{M}^q[\mathcal{M} , \mathcal{M}]$
and $\mathcal{G} = \Phi/\mathcal{L}.$ Since $q \nmid m,$ there is a
semi-direct product decomposition $\mathcal{G} \simeq
\mathcal{M}/\mathcal{L} \rtimes \Phi/\mathcal{M},$ and by the analog
of (\ref{E:Pr15}), we have $\mathcal{M}/\mathcal{L} \simeq
\F_q[\Phi/\mathcal{M}] \oplus \F_q$ as $(\Phi/\mathcal{M})$-modules.
This implies that $C_{\mathcal{G}}(\mathcal{M}/\mathcal{L}) =
\mathcal{M}/\mathcal{L},$ and as $q$ is odd, we obtain the
following:
\begin{equation}\label{E:E3}
\text{for any subgroup} \ \mathcal{G}' \subset \mathcal{G}\ \text{of
index} \ \leqslant 2, \ \text{the centralizer} \
C_{\mathcal{G}}(\mathcal{G}') \ \text{has trivial image in} \
\Phi/\mathcal{M}.
\end{equation}

As before, $F(X)$ will denote a free group on a set $X.$ Let $\phi
\colon \Phi \to F(\bar{u} , \bar{w})$ be the isomorphism such that
$a \mapsto \bar{u}\bar{w},$ $b \mapsto \bar{u},$ and let $\rho
\colon F' = F(u , v , w) \to F(\bar{u} , \bar{w})$ be the
homomorphism defined by $u \mapsto \bar{u},$ $v \mapsto 1,$ $w
\mapsto \bar{w}.$ We will consider $F$ and $F'$ as subgroups of
$\Psi,$ and let $\varkappa_0 \colon \Phi \to \Aut F'$ denote the
homomorphism defined by sending $a , b$ to the restrictions of $\da
, \db$ to $F'.$ Then it follows from Lemma \ref{L:Pr3} that
$\mathrm{Ker}\: \rho$ (which is the normal subgroup of $F'$
generated by $v$) is invariant under $\varkappa_0(\Phi),$ and for $r
\in \Phi,$ the induced action of $\varkappa_0(r)$ on $F(\bar{u} ,
\bar{w})$ coincides with $\Int \phi(r)$ (cf. the proof of
Proposition \ref{P:P2}).

Let $M' = \rho^{-1}(\phi(\mathcal{M}))$ and $L' =
\rho^{-1}(\phi(\mathcal{L})),$ and set $M = M' \cap F$ and $L = L'
\cap F.$ Since $\mathcal{M}$ and $\mathcal{L}$ are normal subgroups
of $\Phi,$ we see that $\phi(\mathcal{M})$ and $\phi(\mathcal{L})$
are normal, hence $\varkappa_0(\Phi)$-invariant, subgroups of
$F(\bar{u} , \bar{w}).$ It follows that $M'$ and $L'$ are normal and
$\varkappa_0(\Phi)$-invariant subgroups of $F',$ and therefore $M$
and $L$ are normal and $\varkappa_0(\Phi)$-, or equivalently,
$\nu_0(\Phi)$-invariant subgroups of $F \cap F';$ besides, $[F \cap
F' : M]$ obviously divides $m.$

Set
$$
S = \bigcap_{g \in F} (gMg^{-1}), \ \ \  T = S \cap F^6[F , F]  \ \
\ \text{and} \ \ \ U = T^q[T , T].
$$
\begin{lemma}\label{L:E1}
{\rm (i)} $U$ is $(\mathrm{Aut}'\: F)$-invariant and is contained in $L;$

\vskip.5mm

\ {\rm (ii)} $[F : U]$ divides $36m^4 \cdot q^{36m^4 + 1};$

\vskip.5mm

{\rm (iii)} if $h \in \Phi$ is such that $\nu_0(h)$ acts on $F/U$ as
$\Int s$ for some $s \in S$ then $h \in \mathcal{M}.$
\end{lemma}
\begin{proof}
(i): By construction, $S$ is a normal $\nu_0(\Phi)$-invariant
subgroup of $F.$ So, it follows from (\ref{E:E5}) that it is
$(\mathrm{Aut}'\: F)$-invariant. Then $T$ and $U$ are also
$(\mathrm{Aut}'\: F)$-invariant. Since $S \subset M,$ we have
$\rho(S) \subset \phi(\mathcal{M}),$ implying that $\rho(U) \subset
\phi(\mathcal{L}),$ and therefore $U \subset L.$

\vskip1mm

(ii): The normalizer $N_F(M)$ contains contains $F \cap F',$ and
since $F/(F \cap F') \simeq (\Z/2\Z)^2,$ we see that $(F \cap F')/S$
embeds in a product of at most four copies of $(F \cap F')/M,$ hence
$[F : S]$ divides $4m^4.$ Taking into account that $$F^6[F , F] =
F^3[F , F] \cap F^2[F , F]$$ and that $S \subset F \cap F' = F^2[F ,
F],$ we conclude that $[F : T]$ divides $36m^4.$ Since $T/U \simeq
\F_q[F/T] \oplus \F_q,$ the index $[F : U]$ divides $36m^4 \cdot
q^{36m^4 + 1}.$

\vskip1mm

(iii): Since $s \in S \subset M',$ there exists $m \in \mathcal{M}$
such that $\phi(m) = \rho(s).$ Set $g = hm^{-1}.$ By (i), $U \subset
L,$ so the action of $\nu_0(h)$ on $(F \cap F')/L'$ coincides with
$\Int s.$ On the other hand, there are isomorphisms
$$
\Phi/\mathcal{L}  \simeq F(\bar{u} , \bar{w})/\phi(\mathcal{L})
\simeq F'/L',
$$
such that for $r \in \Phi,$ the action of $\Int r$ on
$\Phi/\mathcal{L}$ agrees with the action of $\Int \phi(r)$ on
$F(\bar{u} , \bar{w})/\phi(\mathcal{L}),$ and with that of
$\nu_0(r)$ on $F'/L'.$ So, the action of $\nu_0(m)$ on $F'/L'$
coincides with the action of $\Int s,$ and therefore $\nu_0(g)$ acts
on $(F \cap F')L'/L'$ trivially. Since the latter is a subgroup of
index $\leqslant 2$ in $F'/L',$ we conclude that $\Int g$ acts
trivially on a suitable subgroup $\mathcal{G}'$ of $\mathcal{G} =
\Phi/\mathcal{L}$ having index $\leqslant 2,$ and then by
(\ref{E:E3}), $g \in \mathcal{M}$.
\end{proof}

Now, set
$$
K = U \cap K_0 \ \ \text{where} \ \ K_0 = F^4[F , F].
$$
Clearly, $K$ is $(\mathrm{Aut}'\: F)$-invariant, and  since $U
\subset F \cap F' = F^2[F , F],$ it is easy to see that $[F : K]$
divides $144m^4 \cdot q^{36m^4 + 1}.$ We will now show that $K$ is
as required. As we mentioned earlier,
$$
\Gamma[K_0] = \theta^{-1}(SL_2(\Z , 4)) \subset \mathrm{Aut}'\: F,
$$
which implies that $\Gamma[K] \subset \mathrm{Aut}'\: F.$ So, to
complete the verification of (\ref{E:E2}), all we need to show is
that the image of $\pi(\Gamma[K])$ in $\Phi/\mathcal{M}$ is cyclic.
\vskip2mm

Now, let $g \in \Gamma[K].$ According to (\ref{E:E5}), we can write
$g = (\Int f) \cdot \nu_0(h)$ for some $f \in F,$ $h \in \Phi,$ and
then $\nu_0(h)$ acts on $F/K$ as $\Int f^{-1}.$ As we already
mentioned, $\nu_0(h)$ fixes $c = [x , y],$ so $fK \in C_{F/K}(cK),$
and therefore, by Lemma \ref{L:Pr25},  $f \in \la c \ra S.$ Consider
the subgroup $W \subset \mathrm{Aut}(F/K)$ formed by the automorphisms
induced by $\Int s$ with $s \in S,$ and let $\Omega$ be the preimage
of $W$ under the composite map $\Phi
\stackrel{\nu_0}{\longrightarrow} \Aut' F \longrightarrow
\mathrm{Aut}(F/K).$ Set
$$
\Theta = \Gamma[K] \cap \left(\Int F \rtimes \nu_0(\Omega)  \right).
$$
Then our argument shows that the quotient
$\pi(\Gamma[K])/\pi(\Theta)$ is cyclic. On the other hand, by Lemma
\ref{L:E1}(iii),  the image of $\pi(\Theta)$ in $\Phi/\mathcal{M}$
is trivial, and therefore the image of $\pi(\Gamma[K])$  is a cyclic
normal subgroup, as required.
\end{proof}

\vskip2mm

\noindent {\bf Remark 5.3.} 1. If $N$ is a normal subgroup of
$\Gamma$ satisfying $\Int F \subset N \subset \mathrm{Aut}'\: F$ and
such that the quotient $(\mathrm{Aut}'\: F)/N$ has no cyclic normal
subgroups then (\ref{E:E1}), and hence the entire argument, is valid
for $\mathcal{M} = \mathcal{N}$ (in other words, we can set $m =
n$). Then the resulting normal subgroup $K \subset F$ has index
dividing $144n^4 \cdot q^{36n^4 + 1},$ which is much smaller number
than the one given in the statement of the theorem.

\vskip1mm

2. One can somewhat improve the estimation given in the theorem by
choosing for $\mathcal{M}$ the pullback of $\left(\mathcal{N} /
\mathcal{N}^p[\mathcal{N} , \mathcal{N}]\right)^{\Phi/\mathcal{N}}$
(then $\mathcal{N}/\mathcal{M}$ is isomorphic to the augmentation
ideal in $\F_p[\Phi/\mathcal{N}]$) -- this would change $[\Phi :
\mathcal{M}]$ from $n \cdot p^{n + 1}$ to $n \cdot p^{n-1};$ similar
improvements are also possible in the construction of $\mathcal{L}.$
However, with these changes, our construction would become more
cumbersome and less explicit.

\vskip1mm

3. Any explicit procedure yielding a solution of the congruence
subgroup problem for those finite index normal subgroups of $\Gamma$
that contain $\Int F$ leads in fact to its solution for {\it all}
finite index normal subgroups. Indeed, let $N \subset \Gamma$ be an
arbitrary finite index normal subgroup, and let $Q \subset F$ be the
characteristic subgroup that corresponds to $(\Int F) \cap N$ under
the natural isomorphism $F \simeq \Int F.$ Pick a prime $p$ not
dividing $\vert F/Q \vert,$ and set $R = Q^p[Q , Q].$ As we noted
above, $C_{F/R}(Q/R) = Q/R,$ so if $\Int h \in \Gamma[R]$ then $hhomotopy
\in Q.$ By our assumption, we can explicitly find a normal subgroup
$K \subset F$ such that $\Gamma[K] \subset (N \cap \Gamma[R]) \cdot
\Int F.$ We claim that $\Gamma[K \cap R] \subset N.$ Indeed, a given
$g \in \Gamma[K \cap R]$ can be written in the form $g = s \cdot
\Int h$ with $s \in N \cap \Gamma[R]$ and $h \in F.$ Then $\Int h
\in \Gamma[R],$ so $h \in Q$ and therefore $\Int h \in (\Int F) \cap
N.$ It follows that $g \in N,$ as required. We observe, however,
that the described procedure leads to a rather cumbersome estimation
for the index $[F : K \cap R].$

\section{Topological connection}\label{S:T}

In the argument given in \cite{As}, the role of
$\varkappa_0$ in (\ref{E:Pr25}) is played by a homomorphism coming
from the following topological setting. Let $\Sigma^n_g$ be a closed
orientable surface of genus $g$ with $n$ punctures such that $2 - 2g
- n < 0,$ whose fundamental group will be denoted $\pi_1(g , n).$
Consider the configuration space $\mathrm{Conf}_2(\Sigma^n_g)$ of
ordered pairs of distinct points in $\Sigma^n_g.$ It was shown in
\cite{FN} that the natural projection $\mathrm{Conf}_2(\Sigma^n_g)
\to \Sigma^n_g$ is a (locally trivial) fibration with fiber
$\Sigma^{n+1}_g.$ Since $\pi_2(\Sigma^n_g) = 0,$ the exact sequence
of homotopy groups associated to a fibration assumes the form
$$
0 \to \pi_1(g , n+1) \longrightarrow
\pi_1(\mathrm{Conf}_2(\Sigma^n_g)) \longrightarrow \pi_1(g , n) \to
0.
$$
This sequence gives rise to a homomorphism $\rho_{g,n} \colon
\pi_1(g , n) \to \Out \pi_1(g , n +1).$ We will now show that under
appropriate identifications $\Phi \simeq \pi_1(0 , 3)$ and $F'
\simeq \pi_1(0 , 4),$  the homomorphism $\rho_{0,3}$ coincides with
composite map $\Phi \stackrel{\varkappa_0}{\longrightarrow} \Aut F'
\longrightarrow \Out F'.$

  We think of
  \(
    \SurfBot
  \)
  and
  \(
    \SurfTop
  \)
  as the fundamental groups of a plane with two and three punctures,
  respectively; i.e., in terms of punctured spheres,
  we move one of the punctures to infinity.
  We fix generators as follows:
  \[
    \rule[-1cm]{0cm}{2.5cm}
    \setlength{\unitlength}{1cm}%
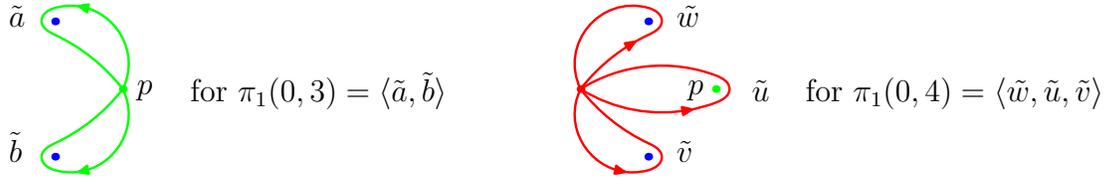
    \begin{picture}(2,0)(0,1)
      \put(0.4,0){\includegraphics[scale=1.0]{loop.10}}
      \put(0,2){\ensuremath{\xgen}}
      \put(0,0.2){\ensuremath{\ygen}}
      \put(1.7,1.1){\ensuremath{\BasePt}}
    \end{picture}
    \text{\ \ \ for\ }
    \SurfBot
    =
    \GroupGenerated{\xgen,\ygen}
    \qquad\qquad
    \begin{picture}(2.7,0)(0,1)
      \put(0,0){\includegraphics[scale=1.0]{loop.20}}
      \put(1.4,2){\ensuremath{\agen}}
      \put(2.4,1){\ensuremath{\bgen}}
      \put(1.4,0.2){\ensuremath{\cgen}}
      \put(1.55,1.1){\ensuremath{\BasePt}}
    \end{picture}
    \text{\ \ \ for\ }
    \SurfTop
    =
    \GroupGenerated{\agen,\bgen,\cgen}
  \]
  Note that the base point $\BasePt$ chosen for $\SurfBot$ is among the
  punctures for $\SurfTop$.

  Of course, the names of
  the generators are chosen suggestively. So, we consider the
  isomorphisms
  \[
    \SurfBot  \longrightarrow \Phi,\qquad
    \xgen  \mapsto  a,\quad
    \ygen  \mapsto  b
  \]
  and
  \[
    \SurfTop  \longrightarrow  F',\qquad
    \agen  \mapsto  w,\quad
    \bgen  \mapsto  u,\quad
    \cgen  \mapsto  v
  \]
\begin{prop}\label{P:T1}
  The action of $\SurfBot$ on $\SurfTop$ induced by the fibration
  \[
    \Sphere{4} \rightarrow \Conf{3} \rightarrow \Sphere{3}
  \]
  is given by:
  \[
    \begin{array}{lll}
    \acts{\xgen}{\agen} = \ainv\binv\agen\bgen\agen, &
    \acts{\xgen}{\bgen} = \ainv\bgen\agen, &
    \acts{\xgen}{\bgen} = \cgen \\
    \acts{\ygen}{\agen} = \agen, &
    \acts{\ygen}{\bgen} = \cgen\bgen\cinv,&
    \acts{\ygen}{\bgen} = \cgen\bgen\cgen\binv\cinv
    \end{array}
  \]
  Consequently, with the above identifications, $\xgen$ acts exactly
  as $\varkappa_0(a) = \dot{\alpha}$, whereas $\ygen$ acts as
  $\operatorname{Int}(u^{-1}v^{-1}) \circ \dot{\beta}$. Therefore, the
  composite map
  \(
    \Phi
    \stackrel{\varkappa_0}{\longrightarrow}
    \operatorname{Aut} F' \longrightarrow \operatorname{Out} F'
  \)
  coincides with $\rho_{0,3}$.
\end{prop}
\begin{proof}
  The action of $\SurfBot$ on $\SurfTop$ induced by the
  fibration
  \[
    \Sphere{4}
    \longrightarrow
    \Conf{4}
    \longrightarrow
    \Sphere{3}
  \]
  is given by the ``push map'': Representing an element of
  $\SurfBot$ as a loop $\LoopBot$ based at $\BasePt$, its effect on
  an element of $\SurfTop$, also given as a loop $\LoopTop$,
  can be seen by pushing
  the base point $\BasePt$ along the inverse of the curve
  $\LoopBot$ and have it drag the loop $\LoopTop$ along.
  (Here the inverse is taken in order to obtain a left action.)

  Using this interpretation, we can read off this action on the generators
  and verify the first claim:
  \[
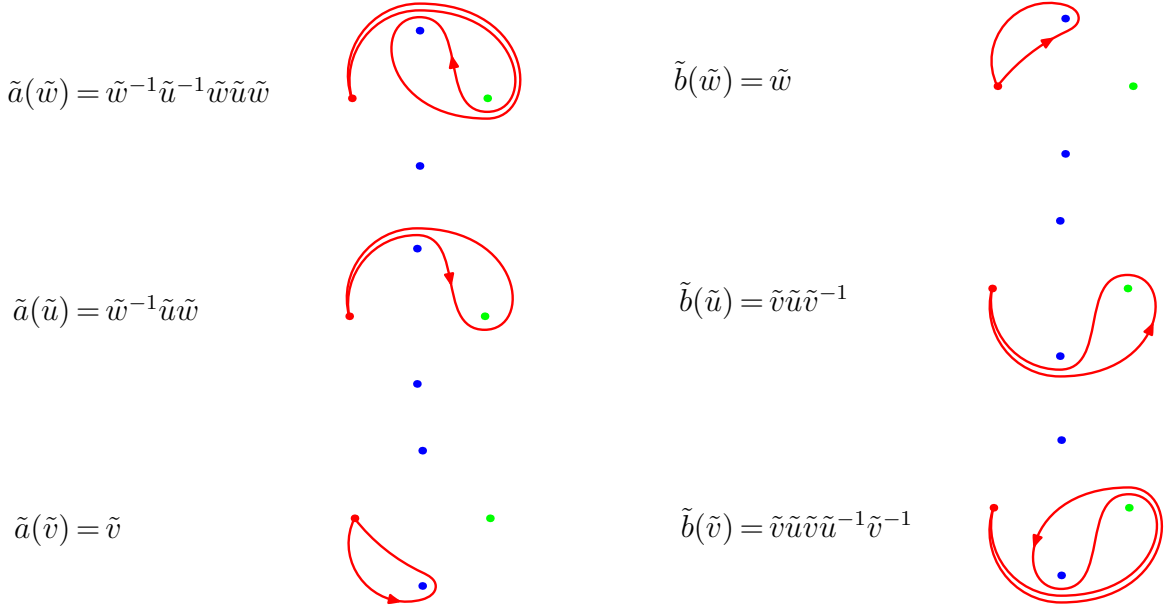
    \begin{array}{r@{\,=\,}l@{\kern 1cm}l}
      \acts{\xgen}{\agen} &\ainv\binv\agen\bgen\agen &
        \raisebox{-0.9cm}{\includegraphics[scale=1.0]{loop.31}} \\[1.5cm]
      \acts{\xgen}{\bgen}& \ainv\bgen\agen &
        \raisebox{-0.9cm}{\includegraphics[scale=1.0]{loop.32}} \\[1.5cm]
      \acts{\xgen}{\cgen}& \cgen &
        \raisebox{-0.9cm}{\includegraphics[scale=1.0]{loop.33}}
    \end{array}
    \qquad\qquad
    \begin{array}{r@{\,=\,}l@{\kern 1cm}l}
      \acts{\ygen}{\agen} & \agen &
        \raisebox{-0.9cm}{\includegraphics[scale=1.0]{loop.51}} \\[1.5cm]
      \acts{\ygen}{\bgen} & \cgen\bgen\cinv &
        \raisebox{-0.9cm}{\includegraphics[scale=1.0]{loop.52}} \\[1.5cm]
      \acts{\ygen}{\cgen} & \cgen\bgen\cgen\binv\cinv &
        \raisebox{-0.9cm}{\includegraphics[scale=1.0]{loop.53}}
    \end{array}
  \]
  Comparing this description to Lemma~\ref{L:Pr3}, we obtain our second claim.
  The last claim follows as $\rho_{0,3}$ is induced by the above action
  of $\SurfBot$ on $\SurfTop$.
\end{proof}

To interpret $\nu_0,$ one needs to consider the following complex
affine algebraic surface
$$
S = \{(s , t , \lambda) \in \C^3 \: \vert \: s^2 = t(t-1)(t -
\lambda), \ \lambda \neq 0, 1 \}.
$$
The projection to the $\lambda$-coordinate gives a fibration $S \to
\PP^1 \setminus \{0, 1, \infty\}$ (where $\PP^1$ is the complex
projective line). The fiber above the value
$\lambda \in \C \setminus \{ 0 , 1 \}$ is
$\EllFiber[\lambda]$ where $\EllCurve[\lambda]$ is the
elliptic curve given by $s^2 = t(t-1)(t - \lambda)$
and $O$ is the point at
infinity on $\EllCurve[\lambda]$
(thus, from the topological point of view, each fiber $\EllFiber[\lambda]$
is a once punctured torus $\Sigma^1_1$).
As above, we have the following exact sequence
$$
0 \to \pi_1(\Sigma_1^1) \longrightarrow \pi_1(S)
\longrightarrow \pi_1(\PP^1 \setminus \{0, 1, \infty\}) \to 0,
$$
which gives rise to a homomorphism $\theta \colon \pi_1(\PP^1
\setminus \{0, 1, \infty\}) \to \Out \pi_1(\Sigma_1^1)$
(cf. \cite{As}, 3.1). We will show in Proposition~\ref{P:T2},
that under
appropriate identifications $\Phi \simeq \pi_1(\PP^1 \setminus \{0,
1, \infty\})$ and $F \simeq \pi_1(\Sigma_1^1),$ this
homomorphism coincides with the composition
\(
  \Phi \xrightarrow{\nu_0} \Aut{F}
  \rightarrow \Out{F}
\)
(Incidentally, this also
provides a proof of the fact, mentioned and used on p. 145 of
\cite{As}, that $\mathrm{Im}\: \theta$ contains the congruence
subgroup $GL_2(\Z , 4),$ which is helpful as the reference given in
{\it loc. cit.} does not seem to contain this fact explicitly.)

Note that $\PP^1 \setminus \{0, 1, \infty\}$ is just the complex
plane punctured at $0$ and $1$. To talk about its fundamental group,
we designate $2$ to be the basepoint. We fix generators of its
fundamental group as follows:
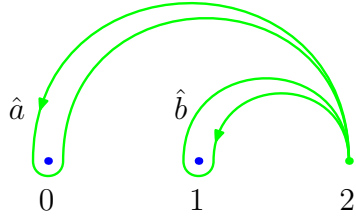
\begin{center}
  \setlength{\unitlength}{1cm}
  \begin{picture}(5,3.1)(-0.5,0)
    \put(-0.125,0){\includegraphics[scale=1.0]{loop.61}}
    \put(0,-0.4){$0$}
    \put(2,-0.4){$1$}
    \put(4,-0.4){$2$}
    \put(-0.4,0.8){$\firstdown$}
    \put(1.8,0.8){$\seconddown$}
  \end{picture}
\end{center}

We also have to deal with the fibers $\EllFiber[\lambda]$.
To do that, we use the projection onto the $t$-coordinate.
It defines a two sheeted cover
\(
  \EllFiber[\lambda] \rightarrow \C
\)
branched above the points $0$, $1$, and $\lambda$. It will be convenient
to be more explicit in the case $\lambda=2$.
The surface $\EllFiber[2]$ is a torus with one puncture (far right in the
pictures). We use $\lambda=2$ as the basepoint $\Green$.
The point $0$ is on the outside left and the point $1$ is
on the inside.
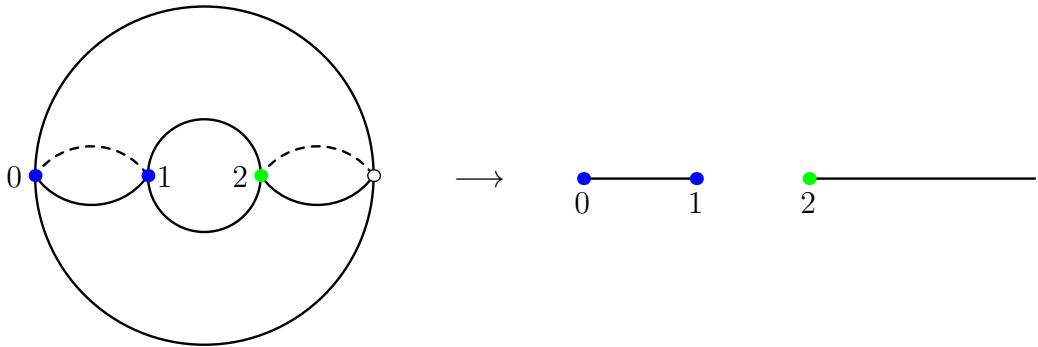
\[
  \raisebox{-2.15cm}{\includegraphics[scale=1.0]{loop.81}}
  \setlength{\unitlength}{1cm}%
  \begin{picture}(0,0)
    \put(-5.0,0){\ensuremath{\BlueOne}}
    \put(-3.0,0){\ensuremath{\BlueTwo}}
    \put(-2.0,0){\ensuremath{\Green}}
  \end{picture}
  \qquad\longrightarrow\qquad
  \includegraphics[scale=1.0]{loop.82}
  \begin{picture}(0,0)
    \put(-6.15,-0.35){\ensuremath{\BlueOne}}
    \put(-4.65,-0.35){\ensuremath{\BlueTwo}}
    \put(-3.15,-0.35){\ensuremath{\Green}}
  \end{picture}
\]
The projection onto the $t$-plane identifies two points if they
are equivalent under the $180$-degree rotation about the axis running
through the colored points. We will think of the
of the surface $\EllFiber[2]$ as a two sheeted cover of the complex
plane branched above ``slits'' (along the real line) from $0$ to $1$ and
from $2$ to $\infty$. The two small circles on the torus are the
preimages of the slits. Thus, the top of the torus is one sheet, and
the bottom of the torus is the other sheet. Moreover, we adopt the
convention that the top-front and bottom-back correspond to the upper
half plane and the top-back and bottom-front project onto the lower
half plane.

Note that $\Out \UpGroup$ is naturally isomorphic to the group
of $\Z$-linear automorphisms of $\FirstHomology$.
Thus, it suffices to study the action of $\DownGroup$ on the
homology of $\EllFiber[2]$. We start by choosing a basis.
In the following
pictures, dashed lines run through the back of the punctured torus.
We also provide the image in the $t$-plane.
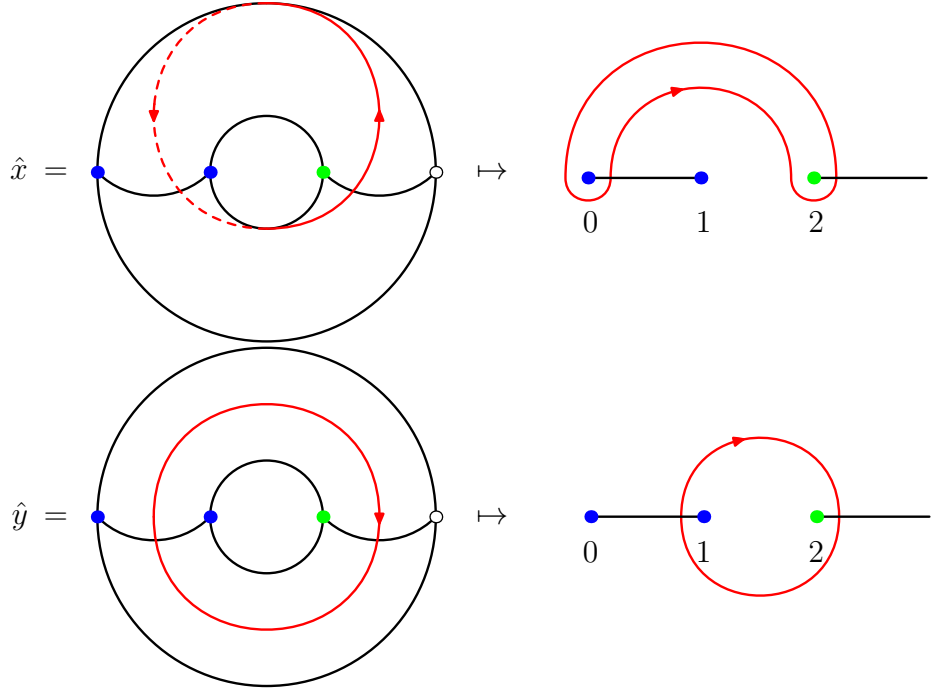
\[
  \begin{array}{r@{\,\,\,=\,\,\,\,}c@{\,\,\,\,\,\,\mapsto\,\,\,\,\,\,\,}c}
  \firstup & \raisebox{-2.2cm}{\includegraphics[scale=1.0]{loop.911}}
  &   \setlength{\unitlength}{1cm}
  \begin{picture}(4,3.0)(-0.5,0.3)
    \put(-0.25,0){\includegraphics[scale=1.0]{loop.912}}
    \put(0,-0.4){$0$}
    \put(1.5,-0.4){$1$}
    \put(3,-0.4){$2$}
  \end{picture}\\
  \secondup & \raisebox{-2.2cm}{\includegraphics[scale=1.0]{loop.921}} &
  \setlength{\unitlength}{1cm}
  \begin{picture}(4,1.0)(-0.5,0.3)
    \put(0,-0.7){\includegraphics[scale=1.0]{loop.922}}
    \put(0,-0.2){$0$}
    \put(1.5,-0.2){$1$}
    \put(3,-0.2){$2$}
  \end{picture}\\
  \end{array}
\]
The pictures in the $t$-plane are ambiguous. The lift $\firstup$
is given by the rule that the crossing of the slit from $0$ to $1$
lifts to a change of sheets from the top to the bottom whereas the
lift $\secondup$ follows the converse convention.

\begin{prop}\label{P:T2}
  The action of $\DownGroup$ on $\FirstHomology$ induced by the
  fibration
  \[
    \EllFiber[2] \longrightarrow S
    \longrightarrow \PP^1 \setminus \{0, 1, \infty\}
  \]
  is given by:
  \[
    \begin{array}{ll}
    \acts{\firstdown}{\firstup} = \firstup,&
    \acts{\firstdown}{\secondup} = \secondup+2\firstup \\
    \acts{\seconddown}{\firstup} = \firstup+2\secondup,&
    \acts{\seconddown}{\secondup} = \secondup \\
    \end{array}
  \]
  We use the obvious identifications of
  $\DownGroup=\GroupGenerated{\firstdown,\seconddown}=
  \GroupGenerated{a,b}=\Phi$ and
  $\FirstHomology=\Z\firstup\oplus\Z\secondup=
  \GroupGenerated{x,y}^{\text{ab}}=F^{\text{ab}}$ as
  indicated by the letters. Thus, with these identifications,
  the homomorphism $\theta$ coincides with
  $\Phi\xrightarrow{\nu_0}\Aut F \rightarrow \Out F$.
\end{prop}
\begin{proof}
  Suppose we have a commutative diagram of continuous maps
  \[
    \xymatrix{
      {\TheAnnulus}
      \ar[r]^{\TheHomotopy}
      \ar[d]_{\pi_2}
      &
      {\TotalSpace}
      \ar[d]^{{\pi_\lambda}}
      \\
      {\TheInterval}
      \ar[r]_{\DownPath}
      &
      {\BaseSpace}
    }
  \]
  I.e., $\TheHomotopy$ is a homotopy of loops in $\TotalSpace$
  each of which runs within a fiber. Assume further, that
  $\DownPath(0)=\DownPath(1)=2$, i.e., $\DownPath$ is a closed
  loop in $\BaseSpace$ based at the basepoint. Then $\TheHomotopy$
  is a homotopy in $\TotalSpace$ connecting two loops in
  $\FiberSpace$. Note that the composite curve
  \(
    \TheHomotopy(0,\cdot) \concat
    \TheHomotopy(\cdot,1) \concat
    \TheHomotopy(0,\cdot)^{\operatorname{rev}}
  \)
  is homotopic to the loop
  \(
    \TheHomotopy(\cdot,0)
  \)
  in $\TotalSpace$. Since $\TheHomotopy(0,\cdot)$ is a lift of
  $\DownPath$, the concatenation represents
  the action of $\DownPath$ on the curve $\TheHomotopy(\cdot,1)$
  by conjugation. Thus
  \(
    \TheHomotopy(\cdot,0) =
    \DownPath( \TheHomotopy(\cdot,1) ),
  \)
  i.e., to compute the effect of a loop $\DownPath$ in
  $\BaseSpace$ on a loop $\UpPath$ in $\FiberSpace$, we have to
  move $\UpPath$ inside $\TotalSpace$ continuously so that (a) at
  each time the curve stays within a single fiber and (b)
  so that the shadow cast by the motion in $\BaseSpace$ traces
  the path $\DownPath^{\operatorname{rev}}$.
  As we are only interested in the action on homology, we can
  consider free loops.

  Now, we have to apply this recipe to the generators. To better
  visualize the process, we use the pictures in the $t$-plane.
  Suppose $\UpPath_\lambda$ is a loop in $\EllFiber[\lambda]$
  whose image in the $t$-plane avoids the branch points $0$,
  $1$, and $\lambda$. Then
  \(
    s^2 = t(t-1)(t-\lambda)
  \)
  is bounded away from $0$ along the compact loop $\UpPath_\lambda$,
  and we can push this
  loop continuously into nearby fibers without changing the image
  in the $t$-plane at all. That allows us to drag $\lambda$ along
  sufficiently small intervals on $\DownPath$; and
  whenever $\lambda$ would run into the $t$-image
  of the curve, we deform the curve within the fiber to avoid
  the collision. This way, we keep $s\neq 0$ at all times
  and thereby do not
  loose control over the lift. Note that deforming a curve in the
  $t$-plane as $\lambda$ traces $\DownPath^{\operatorname{rev}}$
  so as to avoid a collision is the
  same as the ``push map'' from the proof of Proposition~\ref{P:T1}.

  At the end of the process,
  we have to chose between the two possible lifts.
  We will rig the process so that one particular
  point on the initial curve stays put all time (this is easy, we just
  have to make sure that $\lambda$ does not run into its $t$-projection).
  Assume that this point has $t$-coordinate $t_0$.
  From
  \(
    s^2 = t_0(t_0-1)(t_0-\lambda),
  \)
  we see that the lift of this point changes sheets if and only if
  the path that $\lambda$ traces winds around $t_0$ an odd number
  of times. This determines the lift.

  Note that the push map is trivial in the pictures for
  $\firstdown(\firstup)$ and $\seconddown(\secondup)$:
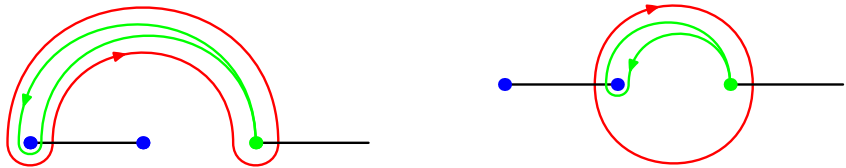
  \[
    \includegraphics[scale=1.0]{loop.931}
    \qquad\qquad
    \includegraphics[scale=1.0]{loop.932}
  \]
  Thus, $\firstdown(\firstup)=\firstup$ and
  $\seconddown(\secondup)=\secondup$.

  We now compute $\firstdown(\secondup)$, i.e., we have to figure out
  the effect of the push map in the following picture:
  
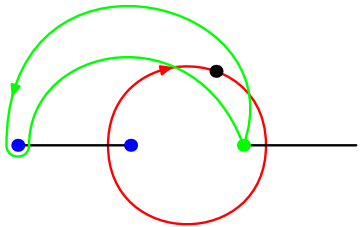
\[
    \includegraphics[scale=1.0]{loop.951}
  \]
  The black dot represents the point that we keep fixed along the
  transformation. Since the path $\firstdown$ surrounds this point,
  we will encounter a change of sheets, which we will have to take into
  account at the end when we determine the lift.

  In the $t$-plane, the result after applying the push map is first
  as follows:
  
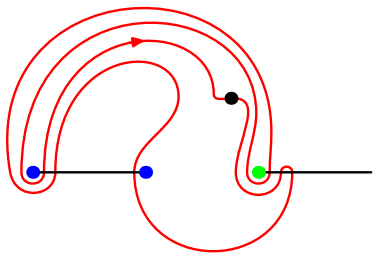
\[
    \includegraphics[scale=1.0]{loop.952}
  \]
  We can simplify this by performing some obvious shortening homotopies,
  still keeping the marked point fixed:
  
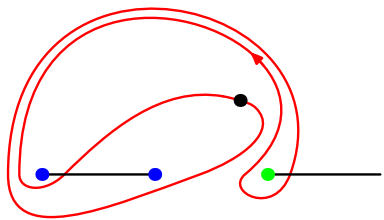
\[
    \includegraphics[scale=1.0]{loop.953}
  \]
  Drawing the lift in $\FiberSpace$, we have to remember that the
  marked point underwent a change of sheets, i.e., it now corresponds
  to a point in the bottom back of the torus:
  
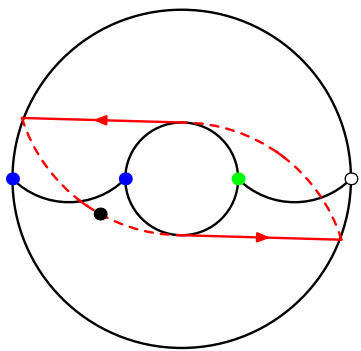
\[
    \includegraphics[scale=1.0]{loop.954}
  \]
  From this picture, we can read off that
  $\firstdown(\secondup)=\secondup+2\firstup$.

  To compute $\seconddown(\firstup)$, we have to determine the action
  of the push map in the following picture:
  
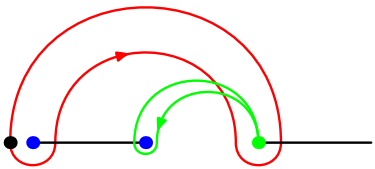
\[
    \includegraphics[scale=1.0]{loop.941}
  \]
  The black dot marks the point that we keep fixed. Note that the generator
  $\seconddown$ does not wind around it. After a simplifying
  homotopy, we get the following lift
  \[
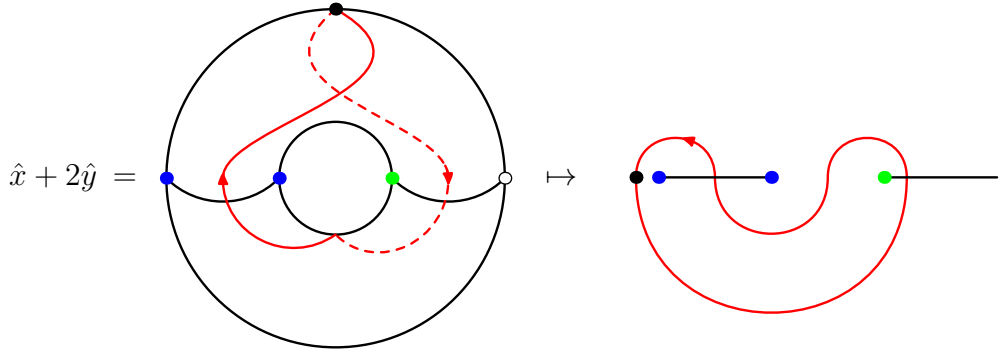
    \begin{array}{r@{\,\,\,=\,\,\,\,}c@{\,\,\,\,\,\,\mapsto\,\,\,\,\,\,\,}c}
      \firstup+2\secondup & \raisebox{-2.2cm}{\includegraphics[scale=1.0]{loop.942}}
      &   \setlength{\unitlength}{1cm}
      \begin{picture}(4,2.0)(-0.5,1.73)
        \put(-0.3,0){\includegraphics[scale=1.0]{loop.943}}
      \end{picture}
    \end{array}
  \]
  which shows that $\seconddown(\firstup)=\firstup+2\secondup$.
\end{proof}

\vskip2mm

To prove the Main Theorem, Asada actually constructs, using
anabelian geometry, certain lifts $$\rho_0^* \colon \widehat{\pi_1(0
, 3)} \to \Aut \widehat{\pi_1(0 , 4)} \ \ \text{and} \ \
\rho_{\mathcal{K}} \colon \widehat{\pi_1(\PP^1 \setminus \{0, 1,
\infty\})} \to \Aut \widehat{\pi_1(E \setminus \{ O \})},$$ for the
homomorphisms that can be identified with the homomorphisms
\begin{equation}\label{E:T10}
\widehat{\pi_1(0 , 3)} \to \Out \widehat{\pi_1(0 , 4)} \ \
\text{and} \ \ \widehat{\pi_1(\PP^1 \setminus \{0, 1, \infty\})} \to
\Out \widehat{\pi_1(E \setminus \{ O \})}
\end{equation}
induced by the homomorphisms of the (discrete) fundamental
groups described
above (in his set-up, $\rho_{\mathcal{K}}$ is defined on a certain
index two subgroup of $\widehat{\pi_1(\PP^1 \setminus \{0, 1,
\infty\})},$ but the passage to this subgroup is not necessary, cf.
the remark prior to Lemma \ref{L:Pr3}). The Main Theorem easily
follows from the fact that $\rho_{\mathcal{K}}$ is injective, and
for this Asada argues that interpreting $\Sigma^3_0$ as $\PP^1
\setminus \{0, 1, \infty\},$ we will have
\begin{equation}\label{E:T11}
\mathrm{Ker}\: \rho_0^* \vert \Theta = \mathrm{Ker}\:
\rho_{\mathcal{K}} \vert \Theta,
\end{equation}
where $\Theta$ is the normal subgroup of $\widehat{\pi_1(0 , 3)} =
\widehat{\pi_1(\PP^1 \setminus \{0, 1, \infty\})}$ generated (as a
closed normal subgroup) by one of the generators of the latter. Then
the injectivity of $\rho_{\mathcal{K}}$ is derived from the injectivity of
$\rho_0^*$ which is provided by Theorem 3B in \cite{As}. Our
Propositions \ref{P:T1} and \ref{P:T2} show that $\varkappa$ and $\nu$
can be taken as the lifts of the homomorphisms in (\ref{E:T10}).
Then the commutativity of (\ref{E:Pr55}) in conjunction with Lemma
\ref{L:PF4} immediately yields that $\mathrm{Ker}\: \varkappa =
\mathrm{Ker}\: \nu,$ which is stronger than (\ref{E:T11}) and yields
the desired injectivity of $\nu$ much quicker.

\bigskip
\paragraph{\textbf{Acknowledgments}}
We are grateful to Pavel Zalesskii and to the anonymous referee for useful comments.
The third-named author would like to 
acknowledge that he was introduced to Asada's paper by Fritz Grunewald.

\bibliographystyle{amsplain}

\end{document}